\documentclass[10pt]{amsart}
\usepackage{fancyhdr}
\usepackage[english]{babel}
\usepackage{amssymb}
\usepackage{mathrsfs}
\usepackage{enumerate}
\usepackage{color}
\usepackage{ifpdf}
\usepackage{tikz}
\usepackage{tikz-cd}
\usetikzlibrary{decorations.pathmorphing,arrows}
\usepackage[initials]{amsrefs}

\usepackage{color}
\usepackage{amsmath}
\usepackage{amssymb,amsfonts}
\usepackage{graphicx}
\usepackage{times,amssymb,amscd}
\usepackage{verbatim}
\usepackage{enumerate}
\usepackage{subcaption}

\DeclareMathOperator\arctanh{arctanh}
\DeclareMathOperator\arccot{arccot}

\newcommand{\vol}{\operatorname{vol}}

\usepackage{lscape}

\newcommand{\bbH}{\mathbb{H}}

\newcommand{\bbR}{\mathbb{R}}

\newcommand{\bh}{\mathbf{h}}
\newcommand{\ba}{\mathbf{a}}

\newcommand{\Bu}{\boldsymbol{u}}

\newcommand{\cD}{\mathcal{D}}
\newcommand{\cF}{\mathcal{F}}

\newcommand{\cT}{\mathcal{T}}
\newcommand{\cC}{\mathcal{C}}
\newcommand{\cB}{\mathcal{B}}

\newcommand{\Hor}{\text{Hor}}

\newtheorem{theorem}{Theorem}
\newtheorem{corollary}{Corollary}
\newtheorem{lemma}{Lemma}
\newtheorem{defn}{Definition}
\newtheorem{rmrk}{Remark}

\newenvironment{definition}{\begin{defn}\normalfont}{\end{defn}}

\usepackage{tikz, tikz-3dplot, pgfplots}
\usepackage{pictexwd,dcpic}

\usepackage{lscape}

\begin{document}

\title{Optimal Horoball Packing Densities for Koszul-type tilings in Hyperbolic $3$-space}

\author{Robert T. Kozma}
\address{Department of Mathematics, Statistics, and Computer Science,
 University of Illinois at Chicago,
Chicago IL 60607 USA \\
Currently: iSpot AI,
15831 NE 8th st,
Bellevue WA 98008 USA\\ 
e-mail: rkozma2@uic.edu}

\author{Jen\H{o}  Szirmai}
\address{Department of Algebra and Geometry, Institute of Mathematics,
Budapest University of Technology and Economics,
M\H{u}egyetem rkp. 3.,
H-1111 Budapest, Hungary\\
e-mail: szirmai@math.bme.hu}
\date{}

\begin{abstract}
We determine the optimal horoball packing densities for Koszul-type Coxeter simplex tilings in hyperbolic $3$-space.  
Using a parametrization of horoballs by the Busemann function and the symmetry of the tilings, we obtain families of packings that attain the universal simplicial density upper bound
\[
d_3(\infty) \;=\; \left( 2 \sqrt{3}\,\Lambda\!\left(\tfrac{\pi}{3}\right) \right)^{-1}
 \;\approx\; 0.853276,
\]
where $\Lambda$ denotes the Lobachevsky function.  
These results show that extremal packing densities in $\mathbb{H}^3$ are realized by multiple explicit Coxeter tilings and are closely tied to special values of $L$-functions and hyperbolic manifold volumes.

\end{abstract}


%%%%%%%%%%%%%%%%%%%%%%%%%%%%%%%%%%%%%%%%%%%
\maketitle
%\tableofcontents
%%%%%%%%%%%%%%%%%%%%%%%%%%%%%%%%%%%%%%%%%%%

%==================================================================%
%                             the main article                               %
%==================================================================%

%%%%%%%%%%%%%%%%%%%%%%%%%%%%%%%%%%%%%%%%%%%
%%%%%%%%%%%%%%%%%%%%%%%%%%%%%%%%%%%%%%%%%%%
\section{Introduction}
%%%%%%%%%%%%%%%%%%%%%%%%%%%%%%%%%%%%%%%%%%%
%%%%%%%%%%%%%%%%%%%%%%%%%%%%%%%%%%%%%%%%%%%

\subsection{Statement of the Result}

In this paper we extend the results of \cite{KSz} to the 23 tetrahedral Koszul simplex groups $\Gamma$ acting on hyperbolic $3$-space $\mathbb{H}^3$ with at least one ideal vertex; see Table \ref{table:simplex_list}.
For these Coxeter tilings we determine the optimal horoball packing densities.

The main invariants appearing in our results are classical special functions.
The \emph{Lobachevsky function} is defined by
\begin{equation}
\Lambda(\theta) \;=\; \frac{1}{2}\,\mathrm{Im}\!\left( \mathrm{Li}_2\!\left( e^{2 i \theta} \right)\right)
 \;=\; - \int_0^{\theta} \ln |2 \sin t| \, dt,
\end{equation}
where $\mathrm{Li}_2$ denotes the dilogarithm.
The $\Lambda$ function admits several alternative descriptions: while Kummer expressed it in terms of the dilogarithm, it can also be written using the polygamma function $\psi_1$.
Its special values are closely connected to values of Dirichlet $L$-functions.  

We also recall \emph{Catalan's constant}
\begin{equation}
C \;=\; L(2,\chi_{-4}) \;=\; \sum_{n=0}^\infty \frac{(-1)^n}{(2n+1)^2},
\end{equation}
the alternating sum of reciprocals of odd squares, which equals one quarter of the volume of the ideal hyperbolic octahedron.  

Another important quantity is the \emph{asymptotic simplicial density} introduced by Kellerhals \cite{K98}:
\begin{equation}
d_3(\infty) \;=\; \left( 2 \sqrt{3}\, \Lambda\!\left(\tfrac{\pi}{3}\right) \right)^{-1}
\;=\; 0.85328\dots.
\end{equation}
This gives the universal upper bound for horoball packing density in $\mathbb{H}^3$, and it is attained by certain Koszul simplex tilings (Theorem~\ref{thm:336}).  

The special value $\Lambda(\tfrac{\pi}{3})$ is especially significant.  
It is related to Gieseking’s constant,
\begin{equation}
G \;=\; 3 \Lambda\!\left(\tfrac{\pi}{3}\right) \;=\; \tfrac{3\sqrt{3}}{4}\, L(2,\chi_{-3}),
\end{equation}
which equals the hyperbolic volume of the Gieseking manifold, the smallest-volume nonorientable cusped hyperbolic $3$-manifold.
Thus our density results can be expressed equivalently in terms of $L(2,\chi_{-3})$.

The following theorems give the optimal packing densities for each commensurability class.

%[3,3,6]
\begin{theorem}
The optimal horoball packing density of Coxeter simplex tilings $\cT_{\Gamma}$, 
$\Gamma \in \Big\{ \overline{V}_3, \overline{Y}_3, \overline{VP}_3, \widehat{PP}_3, \overline{P}_3, \overline{Z}_3, \overline{DV}_3, \overline{DP}_3 \Big\}$
 is $\delta_{opt}(\Gamma) = \left( 2 \sqrt{3}\, \Lambda\!\left(\tfrac{\pi}{3}\right) \right)^{-1} \approx 0.853276$.
 
The optimal horoball packing density of Coxeter simplex tilings $\cT_{\Gamma}$, \\
$\Gamma \in \Big\{ \overline{BV}_3, \overline{BP}_3, \overline{VV}_3 \Big\}$ is $\delta_{opt}(\Gamma) = \frac{2}{5 \sqrt{3}} \left(\Lambda(\frac{\pi}{3})\right)^{-1} \approx 0.682620$.
\label{thm:336}
\end{theorem}

%[3,4,4]
\begin{theorem}
The optimal horoball packing density of Coxeter simplex tilings $\cT_{\Gamma}$, 
$\Gamma \in \Big\{ \overline{R}_3, \overline{O}_3, \widehat{BR}_3, \overline{N}_3, \overline{M}_3, \widehat{RR}_3 \Big\}$
 is $\delta_{opt}(\Gamma) = \frac{3}{4}C^{-1} \approx 0.818808$.
 
\label{thm:344}
\end{theorem}

%[5,3,6]
\begin{theorem}
The optimal horoball packing density of Coxeter simplex tilings $\cT_{\Gamma}$, 
$\Gamma \in \Big\{ \overline{HV}_3, \overline{HP}_3 \Big\}$
 is $\delta_{opt}(\Gamma) \approx 0.550841$.
 
\label{thm:536}
\end{theorem}

%Other nonarithmetic
\begin{theorem}
The optimal horoball packing densities of Coxeter simplex tilings 
$\cT_{\widehat{AV}_3}$, 
$\cT_{\widehat{BV}_3}$, 
$\cT_{\widehat{HV}_3}$, and 
$\cT_{\widehat{CR}_3}$ are respectively
$\delta_{opt}(\widehat{AV}_3) = 0.838825$, 
$\delta_{opt}(\widehat{BV}_3) = 0.747914$, 
$\delta_{opt}(\widehat{HV}_3) = 0.655381$, and
$\delta_{opt}(\widehat{CR}_3) = 0.767195$.

\label{thm:other}
\end{theorem}

Upper bounds for the packing density were published by Kellerhals \cite{K98} using the simplicial density function $d_n(\infty)$. This bound is strict for $n=3$. Combining the results of Theorems \ref{thm:336}--\ref{thm:other}:

\begin{corollary} % make corollary
The optimal horoball packing density for noncompact Coxeter simplex tilings in
$\overline{\bbH}^3$, 
$$d_3(\infty) \;=\; \frac{\sqrt{3}}{2} G^{-1} \;=\; \frac{2}{3}L(2, \chi_{-3})^{-1}$$ 
is realized in tilings by eight simplex groups in the $[3,3,6]$ commensurability class. 
\label{cor:main}
\end{corollary}

\subsection{Relation to Earlier Work}

% maybe define horoball somewhere here, at least in a easy way?
This is the fifth paper in our series determining optimal horoball packing densities in $\bbH^n$ for Koszul-type non-compact Coxeter simplex tilings with $3 \leq n \leq 9$. 
In \cite{KSz}, we showed that the classical horoball packing in $\overline{\bbH}^3$ achieving the B\"or\"oczky-type simplicial packing density upper bound $d_3(\infty)$ (cf. Theorem \ref{t:Boroczky}) is not unique, and demonstrated several new examples using horoballs of {\it different types} -- a notion since made precise through Busemann functions \cite{Busemann, burger2013rigidity, KT}. 
We use isometry invariant Busemann functions to strengthen the notion of \emph{horoball type} with respect to a fundamental domain and to parameterize horoballs centered at $\xi \in \partial \bbH^n$ with respect to a marked reference point $o \in \bbH^n$ (alternatively a reference horoball through $\xi$ and $o$)  in the model of $\overline{\bbH}^n$, see Section \ref{busemann}. This perspective shows that optimal packings cannot be made equivalent by repartitioning, a nontrivial hyperbolic isometry, or some paradoxical construction, and clarifies our earlier results prior to \cite{KT,KSz21}.
In \cite{KSz14,KSz18,KSz21} we considered higher dimensions $ 4 \leq n \leq 9$ where Koszul-type tilings exist.
In \cite{KSz14}, we found seven horoball packings of Coxeter simplex tilings in $\overline{\mathbb{H}}^4$ that yield densities of $5\sqrt{2}/\pi^2 \approx 0.71645$, counterexamples to L. Fejes T\'oth's conjecture for the maximal packing density of $\frac{5-\sqrt{5}}{4} \approx 0.69098$ in his foundational book {\it Regular Figures} \cite[p. 323]{FTL}.
In \cite{KSz18,KSz21} we explicitly constructed the densest known ball packings in $\overline{\mathbb{H}}^n$ for dimensions $6 \le n \le 9$. For example in $\overline{\bbH}^5$ the optimal packing density is $\frac{5}{7 \zeta(3)}$ where $\zeta(\cdot)$ is the Riemann Zeta function.

Let $X$ denote a space of constant curvature, either the $3$-dimensional sphere $\mathbb{S}^3$, 
Euclidean space $\mathbb{E}^3$, or 
hyperbolic space $\overline{\bbH}^3$. An important question of discrete geometry is to find the highest possible packing density in $X$ by congruent non-overlapping balls of a given radius \cite{FFK,G--K}. 
The definition of packing density is critical in hyperbolic space as shown by B\"or\"oczky \cite{B78}, for the standard paradoxical construction see \cite{G--K,R06}. 
The most widely accepted notion of packing density considers the local densities of balls with respect to their Dirichlet--Voronoi cells (cf. \cite{B78,K98}). In order to study horoball packings in $\overline{\mathbb{H}}^3$, we use an extended notion of such local density. 

Let $B$ be a horoball of packing $\cB$, and $P \in \overline{\mathbb{H}}^3$ an arbitrary point. 
Define $d(P,B)$ to be the shortest distance from point $P$ to the horosphere $S = \partial B$, where $d(P,B)\leq 0$  if $P \in B$. The Dirichlet--Voronoi cell $\cD(B,\cB)$ of horoball $B$ is the convex body
\begin{equation}
\cD(B,\cB) = \{ P \in \mathbb{H}^3 | d(P,B) \le d(P,B'), ~ \forall B' \in \cB \}. \notag
\end{equation}
Both $B$ and $\cD$ have infinite volume, so the standard notion of local density is
modified. Let $\xi \in \partial{\mathbb{H}}^3$ denote the ideal center of $B$, and take its boundary $\mathcal{S}$ to be the one-point compactification of Euclidean plane.
Let $B_C(r) \subset S$ be the Euclidean ball with center $C \in \mathcal{S} \setminus \{\xi\}$.
Then $\xi$ and $B_C(r)$ determine a convex cone 
$\cC(r) = cone_{\xi}\left(B_C(r)\right) \in \overline{\mathbb{H}}^3$ with
apex $\xi$ consisting of all hyperbolic geodesics passing through $B_C(r)$ with limit point $\xi$. The local density $\delta_3(B, \cB)$ of $B$ to $\cD$ is defined as
\begin{equation}
\delta_n(\cB, B) =\varlimsup\limits_{r \rightarrow \infty} \frac{vol_{\bbH^3}(B \cap \cC(r))} {vol_{\bbH^3}(\cD \cap \cC(r))}. \notag
\end{equation}
This limit is independent of the choice of center $C$ for $B_C(r)$.

In the case of periodic ball or horoball packings, this local density defined above extends to the entire hyperbolic space via its symmetry group, and 
is related to the simplicial density function (defined below) that we generalized in \cite{Sz12,Sz12-2}.
In this paper we shall use such definition of packing density (cf. Section \ref{s:lemmas}).  

A Coxeter simplex is a top dimensional simplex in $X$ with dihedral angles either integral submultiples of $\pi$ or zero. 
The group generated by reflections on the sides of a Coxeter simplex is a Coxeter simplex reflection group. 
Such reflections generate a discrete group of isometries of $X$ with the Coxeter simplex as the fundamental domain. 
Hence the groups give regular tessellations of $X$ if the fundamental simplex is characteristic. The Coxeter groups are finite for $\mathbb{S}^3$, and infinite for $\mathbb{E}^3$ or $\overline{\mathbb{H}}^3$. Tables of the hyperbolic 3-simplex groups can be found in \cite[\S10.4]{EGM, VS}.

There are non-compact Coxeter simplices in $\overline{\mathbb{H}}^n$ with ideal vertices on $\partial \mathbb{H}^n$, however only for dimensions $2 \leq n \leq 9$; furthermore, only a finite number exist in dimensions $n \geq 3$. 
Johnson {\it et al.} \cite{JKRT} found the volumes of all Coxeter simplices in hyperbolic $n$-space. 
Such simplices are the most elementary building blocks of hyperbolic manifolds,
the volume of which is an important topological invariant. 

In $n$-dimensional space $X$ of constant curvature
 $(n \geq 2)$, define the simplicial density function $d_n(r)$ to be the density of $n+1$ mutually tangent balls of radius $r$ in the simplex spanned by their centers. L.~Fejes T\'oth and H.~S.~M.~Coxeter
conjectured that the packing density of balls of radius $r$ in $X$ cannot exceed $d_n(r)$.
Rogers \cite{Ro64} proved this conjecture in Euclidean space $\mathbb{E}^n$.
The $2$-dimensional spherical case was settled by L.~Fejes T\'oth \cite{FTL}, and B\"or\"oczky \cite{B78} gave a proof for the extension:
\begin{theorem}[K.~B\"or\"oczky]
In an $n$-dimensional space of constant curvature, consider a packing of spheres of radius $r$.
In the case of spherical space, assume that $r<\frac{\pi}{4}$.
Then the density of each sphere in its Dirichlet--Voronoi cell cannot exceed the density of $n+1$ spheres of radius $r$ mutually
touching one another with respect to the simplex spanned by their centers.
\label{t:Boroczky}
\end{theorem}

In hyperbolic 3-space, the monotonicity of $d_3(r)$ was established by B\"or\"oczky and Florian
in \cite{B--F64}, while Marshall \cite{Ma99} 
showed that for sufficiently large dimensions $n$ 
the function $d_n(r)$ is strictly increasing in $r$. Kellerhals \cite{K98} further proved that $d_n(r)<d_{n-1}(r)$ and that, in the cases considered by Marshall, the local density of each ball in its Dirichlet--Voronoi cell is bounded above by the simplicial horoball density $d_n(\infty)$. Theorem \ref{t:Boroczky} is extended to the case of horoballs in \cite[\S 6]{B78} as a remark.

The simplicial upper bound
 \[
d_3(\infty) = \Bigl(1+\frac{1}{2^2}-\frac{1}{4^2}-\frac{1}{5^2}+\frac{1}{7^2}+\frac{1}{8^2}--++\cdots\Bigr)^{-1} \approx 0.85327
 \]
cannot be realized by packings of regular balls, but it is attained by horoball packings associated with the ideal Coxeter simplex tiling of $\overline{\mathbb{H}}^3$ with Schl\"afli symbol $[3,3,6]$. In this case the horoball centers lie at the ideal vertices of the regular simplex on $\partial\overline{\mathbb{H}}^3$. 

Such bounds have important applications beyond packing problems. 
In dimension three, B\"or\"oczky-type estimates for horoball packings have been used to study the volume of cusped hyperbolic manifolds \cite{A87,M86,ACS,MaM}. 
If $\Gamma$ is a discrete torsion-free subgroup of isometries, then $\bbH^3/\Gamma$ is a cusped hyperbolic manifold, and cusp neighborhoods lift to horoball packings in the universal cover. 
A manifold with a single cusp has a well-defined maximal cusp neighborhood, while manifolds with multiple cusps admit a range of non-overlapping cusp neighborhoods whose boundaries meet tangentially; these give rise to different horoball types in $\bbH^3$. 
An important application is Adams' proof that the Gieseking manifold is the unique noncompact hyperbolic $3$-manifold of minimal volume \cite{A87}. Kellerhals then used the B\"or\"oczky-type bounds to estimate volumes of higher dimensional hyperbolic manifolds \cite{K98_2}.  

In \cite{KSz} we proved that the classical horoball packing configuration in $\mathbb{H}^3$ achieving the B\"or\"oczky-type upper bound  is not unique. We gave several examples of different regular horoball packings using horoballs of different types, that is horoballs that have different relative densities with respect to the fundamental domain, that yield the B\"or\"oczky--Florian-type simplicial upper bound \cite{B--F64}. 
We also note that in related work the second-named author studied simply truncated Coxeter orthoschemes in $\bbH^3$, where the simplicial density bound is not realized \cite{SzY}.

%%%%%%%%%%%%%%%%%%%%%%%%%%%%%%%%%%%%%%%%%%%
%%%%%%%%%%%%%%%%%%%%%%%%%%%%%%%%%%%%%%%%%%%
\section{Preliminaries}
%%%%%%%%%%%%%%%%%%%%%%%%%%%%%%%%%%%%%%%%%%%
%%%%%%%%%%%%%%%%%%%%%%%%%%%%%%%%%%%%%%%%%%%

We use the projective Cayley--Klein model of hyperbolic geometry to preserves lines and convexity for the packing of simplex tilings with convex fundamental domains. 
In this section we review some key concepts, for a general discussion of the projective models of Thurston geometries see \cite{Mol97,MSz}.

%%%%%%%%%%%%%%%%%%%%%%%%%%%%%%%%%%%%%%%%%%%
\subsection{The Projective Model of $\overline{\mathbb{H}}^3$}
\label{ss:projective_model}
%%%%%%%%%%%%%%%%%%%%%%%%%%%%%%%%%%%%%%%%%%%

We begin with $\mathbb{R}^4$ equipped with the Lorentzian quadratic form
\[
\langle\mathbf{x},\mathbf{y}\rangle \;=\; -x_0 y_0 + x_1 y_1 + x_2 y_2 + x_3 y_3 ,
\]
for $\mathbf{x}=(x_0,x_1,x_2,x_3)$ and $\mathbf{y}=(y_0,y_1,y_2,y_3)$.
Nonzero vectors $\mathbf{x}\in \mathbb{R}^4$ represent points in projective space
\[
\mathbb{P}^3=\mathbb{P}(\mathbb{R}^{4}),
\]
equipped with the quotient topology of the projection 
\[
\Pi:\;\mathbb{R}^4 \setminus \{\mathbf{0}\} \longrightarrow \mathbb{P}^3, \quad 
\mathbf{x} \mapsto [\mathbf{x}].
\]

We partition $\mathbb{R}^4$ into
\[
Q_+ = \{\mathbf{v} \mid \langle \mathbf{v},\mathbf{v}\rangle > 0\},\quad
Q_0 = \{\mathbf{v} \mid \langle \mathbf{v},\mathbf{v}\rangle = 0\},\quad
Q_- = \{\mathbf{v} \mid \langle \mathbf{v},\mathbf{v}\rangle < 0\}.
\]
The hyperbolic $3$–space is realized as
\[
\mathbb{H}^3 = \Pi(Q_-), \qquad 
\partial\mathbb{H}^3 = \Pi(Q_0), \qquad 
\overline{\mathbb{H}}^3 = \mathbb{H}^3 \cup \partial \mathbb{H}^3.
\]
Points in $\Pi(Q_0)$ are  \emph{ideal} or  \emph{boundary points}, while points in $\Pi(Q_+)$ are the \emph{outer points}.

\medskip

\noindent\textbf{The Cayley–Klein section.}
For computations we use the Cayley–Klein model, obtained by intersecting $\mathbb{P}^3$ with the affine chart $x_0=1$. Thus a projective point $[\mathbf{x}]$ can be represented uniquely as
\[
\mathbf{x} = (1,x_1,x_2,x_3) \in \mathbb{R}^4 .
\]

\medskip

\noindent\textbf{Polarity.}
Two points $[\mathbf{x}], [\mathbf{y}] \in \mathbb{P}^3$ are called \emph{conjugate} if $\langle \mathbf{x},\mathbf{y} \rangle=0$.
The set of all points conjugate to $[\mathbf{x}]$ forms a hyperplane in $\mathbb{P}^3$,
\[
\operatorname{pol}([\mathbf{x}])=\{[\mathbf{y}]\in\mathbb{P}^3 \mid \langle\mathbf{x},\mathbf{y}\rangle=0\}.
\]
In this way the quadratic form $B$ induces a natural duality between points and hyperplanes in $\mathbb{P}^3$.  
Incidence of a point $[\mathbf{x}]$ with a hyperplane given by a pole $\mathbf{u}$ is expressed simply as
\[
\langle\mathbf{x},\mathbf{u}\rangle = 0.
\]

\medskip

\noindent\textbf{Hyperplanes and polyhedra.}
Let $P \subset \overline{\mathbb{H}}^3$ be a polyhedron bounded by finitely many hyperplanes $H^i$, each determined by a pole vector $\mathbf{b}^i \in \mathbb{R}^4$ normalized so that
\[
\langle\mathbf{b}^i,\mathbf{b}^i\rangle = 1.
\]
We assume each $\mathbf{b}^i$ points inward to $P$. Then
\[
H^i = \{\, [\mathbf{x}] \in \mathbb{H}^3 \mid \langle\mathbf{x},\mathbf{b}^i\rangle = 0 \,\}.
\]

Throughout, $P$ is an acute–angled polyhedron with proper or ideal vertices.  
The Gram matrix of $P$ is
\[
G(P) = \big(\langle\mathbf{b}^i,\mathbf{b}^j\rangle\big)_{i,j}
\]
which is symmetric of signature $(1,3)$. Its entries satisfy $\langle\mathbf{b}^i,\mathbf{b}^i\rangle=1$ and
\[
\langle\mathbf{b}^i,\mathbf{b}^j\rangle \;=\;
\begin{cases}
0 & \text{if $H^i \perp H^j$,}\\
-\cos(\alpha^{ij}) & \text{if $H^i,H^j$ meet along an edge at angle $\alpha^{ij}$,}\\
-1 & \text{if $H^i,H^j$ are parallel,}\\
-\cosh(l^{ij}) & \text{if $H^i,H^j$ admit a common perpendicular of length $l^{ij}$.}
\end{cases}
\]

The Coxeter diagram $\Sigma(P)$ encodes this data: each node corresponds to a hyperplane $H^i$, and nodes are joined if $H^i$ and $H^j$ are not orthogonal.  
Edges carry the label $k$ if $\alpha^{ij}=\pi/k$, and unlabeled edges indicate $\alpha^{ij}=\pi/3$. Examples are listed in Table~\ref{table:simplex_list}.

\medskip

\noindent\textbf{Distance and projections.}
The hyperbolic metric in the Cayley–Klein model has curvature $K=-1$.  
For proper points $[\mathbf{x}],[\mathbf{y}]\in\mathbb{H}^3$ (with $x_0=y_0=1$), the distance is given by
\begin{equation}
\cosh d([\mathbf{x}],[\mathbf{y}]) = 
\frac{-\langle\mathbf{x},\mathbf{y}\rangle}{\sqrt{\langle\mathbf{x},\mathbf{x}\rangle \, \langle\mathbf{y},\mathbf{y}\rangle}} .
\label{prop_dist}
\end{equation}
The orthogonal projection of a point $[\mathbf{x}]$ to the hyperplane with pole $\mathbf{u}$ is
\[
\mathbf{y} \;=\; \mathbf{x} - \frac{\langle\mathbf{x},\mathbf{u}\rangle}{\langle\mathbf{u},\mathbf{u}\rangle}\,\mathbf{u}.
\]

%%%%%%%%%%%%%%%%%%%%%%%%%%%%%%%%%%%%%%%%%%%
\subsection{Horospheres and Horoballs in $\overline{\mathbb{H}}^3$}
%%%%%%%%%%%%%%%%%%%%%%%%%%%%%%%%%%%%%%%%%%%
A \emph{horosphere} in $\overline{\mathbb{H}}^3$ is defined as a limit of metric spheres in $\mathbb{H}^3$ whose centers converge to an ideal point. More precisely, fix $x \in \mathbb{H}^3$ and $\xi \in \partial\mathbb{H}^3$, and let $c_t \in \mathbb{H}^3$ be a point on the geodesic ray $[x,\xi)$ with $c_t \to \xi$ as $t \to \infty$. The sequence of spheres in $\mathbb{H}^3$ centered at $c_t$ and passing through $x$ converges (in the Hausdorff sense) to the horosphere $\Hor_{\xi}(x)$ tangent to $\partial\mathbb{H}^3$ at $\xi$.  
A \emph{horoball} is the union of all metric balls in the above family, i.e. the closed region bounded by the horosphere $\Hor_{\xi}(x)$ and containing the geodesic ray $[x,\xi)$. 

\medskip

\noindent\textbf{Equation of a horosphere.}
Choose coordinates in $\mathbb{P}^3$ with standard basis vectors $\mathbf{a}_i$, $0 \leq i \leq 3$, such that the Cayley–Klein ball model of $\overline{\mathbb{H}}^3$ is centered at 
\[
O=(1,0,0,0),
\]
and orient it with the ideal point $\xi \in \partial\mathbb{H}^3$ at
\[
A_0 = (1,0,0,1).
\]
The boundary of $\mathbb{H}^3$ in this model is the quadric
\[
-x_0^2 + x_1^2 + x_2^2 + x_3^2 = 0,
\]
and the hyperplane $x_0-x_3=0$ is tangent to this quadric at $A_0$.  

The equation of a horosphere centered at $\xi=A_0$ and passing through an interior point $S=(1,0,0,s)$ is then a quadratic form of the type
\begin{equation}
0 = \lambda\big(-x_0^2 + x_1^2+x_2^2+x_3^2\big) + \mu(x_0-x_3)^2 .
\label{horosphere_eq}
\end{equation}
Evaluating at $S$ determines $\lambda,\mu$, giving
\[
\frac{\lambda}{\mu} = \frac{1-s}{1+s}, \qquad s \neq \pm 1.
\]
Thus the explicit equation of the horosphere in projective coordinates is
\begin{equation}
(s-1)\Big(-x_0^2 + \sum_{i=1}^3 x_i^2\Big) - (1+s)(x_0-x_3)^2 = 0.
\label{eq:horosphere}
\end{equation}

\noindent
The real parameter $s \in (-1,1)$ determines which horosphere in the one–parameter family centered at $A_0$ is chosen: for each $s$ the horosphere passes through the interior point $S=(1,0,0,s)$. 
We shall refer to $s$ as the \emph{$s$–parameter} of the horosphere (or horoball). 
Different values of $s$ yield distinct, non–congruent horoballs centered at $A_0$, while $s \to 1$ (resp.~$s \to -1$) corresponds to the horosphere approaching the boundary of the model.

In the Cayley–Klein model (the affine chart $x_0=1$), horospheres appear as Euclidean ellipsoids tangent to the boundary sphere of the model, and horoballs are the corresponding solid ellipsoids intersected with $\mathbb{H}^3$.

\medskip

\noindent\textbf{Isometries and horoball types.}
For any two horoballs $B,B'\subset\overline{\mathbb{H}}^3$ there exists an isometry $g \in \mathrm{Isom}(\overline{\mathbb{H}}^3)$ with action $g.B=B'$.  
Nevertheless, in the context of packings it is often convenient to distinguish between horoballs relative to a fixed fundamental domain. Following~\cite{Sz12-2}, we say two horoballs are of the \emph{same type} (or are \emph{equipacked}) if their local packing densities with respect to a given Coxeter cell are equal; otherwise they are of \emph{different type}.  

For instance, the family of horoballs centered at $A_0$ and passing through $S=(1,0,0,s)$ with varying $s \in (-1,1)$ contains infinitely many types relative to a fixed Coxeter simplex with ideal vertex at $A_0$. Thus the set of horoballs centered at the same ideal vertex form a  one-parameter family.

\medskip

\noindent\textbf{Bolyai’s formulas.}
The intrinsic geometry of a horosphere is Euclidean. The hyperbolic length of a horospherical arc of Euclidean chord length $x$ is
\begin{equation}
L(x)=2\sinh\!\left(\tfrac{x}{2}\right).
\label{eq:horo_dist}
\end{equation}
If $A$ is a polygonal region on a horosphere of Euclidean area $\mathcal{A}$, then the volume of the horoball sector $\mathcal{H}(A)$ bounded by $A$ and the geodesic segments from $A$ to the horoball center is
\begin{equation}
\operatorname{vol_{\bbH^3}}(\mathcal{H}(A)) = \tfrac{1}{2}\,\mathcal{A}.
\label{eq:bolyai}
\end{equation}
These results, already known to J.~Bolyai in the 19th century, can also be found in modern expositions of hyperbolic geometry (see e.g.~\cite{Rat} or \cite{B78}).

%%%%%%%%%%%%%%%%%%%%%%%%%%%%%%%%%%%%%%%%%
\subsection{The Busemann function in $\overline{\bbH}^3$}
\label{busemann}
%%%%%%%%%%%%%%%%%%%%%%%%%%%%%%%%%%%%%%%%%

The Busemann function (cf.~\cite{Busemann, burger2013rigidity, KT}) on $\overline{\bbH}^3$ is the map
\[
\beta: \bbH^3 \times \bbH^3 \times  \partial \bbH^3 \longrightarrow \bbR, 
\qquad \beta(x,y,\xi)=\lim_{z \rightarrow \xi} \left(d(x,z) - d(y,z)\right),
\]
where the limit $z \to \xi$ is taken along any geodesic in $\bbH^3$ ending at boundary point $\xi$.  
The Busemann function satisfies antisymmetry and the cocycle property:  
\[
\beta(x, x, \xi) = 0, \quad 
\beta(x, y,\xi) = - \beta(y, x,\xi), \quad 
\beta(x, y,\xi) + \beta( y, z,\xi) = \beta(x, z,\xi),
\]
for all $x,y,z \in \bbH^3$, and it is invariant under the action of $\text{Isom}(\bbH^3)$.  

A horosphere centered at $\xi$ through $o$ is the level set
\[
\Hor_{\xi}(o) = \{x \in \bbH^3 \mid \beta(x, o,\xi) = 0 \},
\]
while a horoball is the corresponding sublevel set
\[
\mathring{\Hor}_{\xi}(o) = \{x \in \bbH^3 \mid \beta(x, o,\xi) \le 0 \}.
\]
The space of all horospheres $\Hor(\bbH^3)$ carries an $\bbR$-fibration 
\[
h: \Hor(\bbH^3) \longrightarrow \partial \bbH^3, \qquad \Hor_{\xi}(o) \mapsto \xi,
\]
so that the Busemann function gives an oriented distance between two concentric horospheres $\Hor_{\xi}(o_1)$ and $\Hor_{\xi}(o_2)$.  

\begin{lemma}
\label{lemma:s-parameter}
Fix reference point $o=(1,0,0,0) \in \bbH^3$ and boundary point $\xi=(1,0,0,1) \in \partial \bbH^3$.  
For $x \in \bbH^3$, the associated horosphere $\Hor_{\xi}(x)$ has {\em $s$-parameter}
\[
s(x) \;=\; \tanh\!\big(\beta(o,x,\xi)\big).
\]
\end{lemma}
\noindent
\begin{proof}
Follows from the definitions while taking $z_t$ along the geodesic ray $[o,\xi)$.
Working in the Cayley--Klein model (affine chart $x_0=1$).  
Let $S=(1,0,0,s) \in \Hor_{\xi}(x)$ with $s\in(-1,1)$ and $z_t=(1,0,0,t)$ with $t\in(-1,1)$ tending to $\xi$ as $t\rightarrow 1$.
Using the distance formula \eqref{prop_dist} we compute
\[
\cosh d(S,z_t)=\frac{1-st}{\sqrt{(1-s^2)(1-t^2)}},
\qquad
\cosh d(o,z_t)=\frac{1}{\sqrt{1-t^2}}.
\]
Hence, as $t\rightarrow 1$,
\[
d(S,z_t)-d(o,z_t)
=\operatorname{arcosh}\!\left(\frac{1-st}{\sqrt{(1-s^2)(1-t^2)}}\right)
-\operatorname{arcosh}\!\left(\frac{1}{\sqrt{1-t^2}}\right).
\]
Using $\operatorname{arcosh}u=\log\!\big(u+\sqrt{u^2-1}\big)$ and the asymptotic
$\operatorname{arcosh}u=\log(2u)+o(1)$ as $u\to\infty$, we obtain
\[
\lim_{t\rightarrow 1}\big(d(S,z_t)-d(o,z_t)\big)
=\log\!\left(\frac{1-st}{\sqrt{1-s^2}}\right)\Big|_{t=1}
=\tfrac12\log\!\left(\frac{1-s}{1+s}\right).
\]
By definition,
\(
\beta(S,o,\xi)=\lim_{t\rightarrow 1}\big(d(S,z_t)-d(o,z_t)\big)
=\tfrac12\log\!\left(\frac{1-s}{1+s}\right),
\)
so
\[
\beta(o,S,\xi)=-\beta(S,o,\xi)=\tfrac12\log\!\left(\frac{1+s}{1-s}\right)
=\operatorname{arctanh}(s).
\]
Hence for any $x \in \Hor_{\xi}(x)$, \(s(x)=\tanh\!\big(\beta(o,x,\xi)\big)\) as claimed.
\end{proof}

\noindent
Thus the $s$-parameter measures horospheres centered at $\xi$ relative to the reference horosphere $\Hor_{\xi}(o)$, and coincides (up to scaling) with the projective $s$-parameter introduced above in \S\ref{ss:projective_model}.

\medskip
A choice of basepoint $o \in \bbH^3$ provides a trivialization of the fibration 
according to the diagram
\begin{center}
\begin{tikzcd}
\Hor(\bbH^3) \arrow{r}{\varphi_o} \arrow{d}[swap]{h} & \partial \bbH^3 \times \bbR \arrow{ld}{\pi}  \\
 \partial \bbH^3
\end{tikzcd}
\end{center}
where $\Hor_{\xi}(x) \mapsto (\xi, \beta(o,x,\xi))$.  
An isometry $g \in \text{Isom}(\bbH^3)$ acts on a horosphere as an additive cocycle:
\begin{align*}
g.\Hor_{\xi}(x) = \Hor_{g.\xi}(gx) 
&\mapsto \big( g\xi, \beta(o, gx, g\xi)\big) 
   = \big( g\xi, \beta(g^{-1}o, x, \xi)\big) \\
&= \big( g\xi, \beta(o, x,\xi) + \beta(g^{-1}o, o, \xi)\big).
\end{align*}
Letting $\hat s = \operatorname{arctanh}(s)$, the action in the trivialization is
\[
g(\xi,\hat s) = \big(g\xi, \hat s+\beta(g^{-1}o,o,\xi)\big).
\]

\medskip
In summary, Busemann functions describe horoballs and their $s$-parameters in a way that is invariant under isometries of $\overline{\bbH}^3$, and hence provide a natural tool for studying packing configurations relative to a marked basepoint $o$.

%%%%%%%%%%%%%%%%%%%%%%%%%%%%%%%%%%%%%%%%%%%
%%%%%%%%%%%%%%%%%%%%%%%%%%%%%%%%%%%%%%%%%%%
\section{Packing Density in the Projective Model}
\label{s:lemmas}
%%%%%%%%%%%%%%%%%%%%%%%%%%%%%%%%%%%%%%%%%%%
%%%%%%%%%%%%%%%%%%%%%%%%%%%%%%%%%%%%%%%%%%%

In this section we define packing density and collect three Lemmas used in Section \ref{s:densities} to find the optimal packing densities for the Koszul simplex tilings. 

Let $\cT$ be a Coxeter tiling of $\overline{\mathbb{H}}^3$ \cite{JKRT2}. The symmetry group of a Coxeter tiling $\Gamma_\cT$ 
 contains its Coxeter group, and an isometry between two cells in $\cT$ preserves the tiling.
Any simplex cell of $\cT$ acts as a fundamental domain $\cF_{\cT}$
of $\Gamma_{\cT}$, and the Coxeter group is generated by reflections on the $2$-dimensional faces of $\cF_{\cT}$. 
We restrict to noncompact (Koszul-type) Coxeter simplices, i.e. simplices with one or more ideal vertex, so that the orbifold $\mathbb{H}^3/\Gamma_{\cT}$ has at least one cusp. Table \ref{table:simplex_list} lists the 23 Koszul-type Coxeter simplices in $\overline{\mathbb{H}}^3$ and their hyperbolic volumes. For a detailed discussion of the volume formulae the hyperbolic Coxeter simplices of higher dimensions $n \geq 3$, see Johnson {\it et al.} \cite{JKRT}. 

\begin{definition}[Horoball packing associated to a Coxeter tiling]
\label{def:horoball_packing}
Let $\cT$ be a Coxeter tiling with fundamental simplex $\cF_{\cT}$ and set of ideal vertices $\{v_1, \dots, v_m\}$. A \emph{horoball packing associated to} $\cT$ is a countable $\Gamma_{\cT}$-invariant  collection of horoballs $\cB_{\cT}$ which restricts to the set $\cB_{\cF_{\cT}} = \{B_1, \dots, B_m\}$ on $\cF_{\cT}$such that:
\begin{enumerate}
\item each $B_i$ is centered at $v_i$;
\item horoballs do not overlap, i.e. interiors are pairwise disjoint;
\item each $B_i$ is contained in the half-space bounded by the facet of $\cF_{\cT}$ opposite $v_i$.
\end{enumerate} 
\end{definition}
\noindent
The third condition ensures that the configuration extends consistently by $\Gamma_{\cT}$ to a packing $\cB_{\cT}$ of all $\overline{\mathbb H}^3$, in other words $\cB_{\cT}$ is invariant under the actions of $\Gamma_\cT$. 

The density of such a packing is defined by
\begin{equation}
\delta(\mathcal{B}_{\cT})\;=\;\frac{\sum_{i=1}^m vol_{\mathbb{H}^3}(B_i \cap \cF_{\cT})}{vol_{\mathbb{H}^3}(\cF_{\cT})}.
\label{eq:density}
\end{equation}

We allow horoballs of different types at each ideal vertex of the tiling. The optimal horoball packing density is
\begin{equation}
\delta_{opt}(\cT) \;=\; \sup\limits_{\mathcal{B}_{\cT} \text{~packing}} \delta(\mathcal{B}_{\tau}).
\label{eq:opt_density}
\end{equation}

Now let $\cF_{\Gamma}$ denote the simplicial fundamental domain of Coxeter tiling $\cT_{\Gamma}$ with vertex set $\{A_i\}_{i=0}^3 \in \mathbb{P}(\bbR^4)$, where $A_0=(1, 0, 0, 1)$ is ideal and $A_1 = o = (1, 0, 0, 0)$ is the center of the model. The remaining vertices $A_2, A_3$ are chosen according to the Coxeter diagram of $\cF_{\Gamma}$ in Table \ref{table:simplex_list} (see Tables \ref{table:data_336_main}--\ref{table:other}). We write $\Bu_i$ for the facet opposite to vertex $A_i$.

Lemma \ref{lem:loc} gives the local packing density in the case of a single ideal vertex.
Packing density is maximized by the largest horoball type admissible in cell $\cF_{\Gamma}$ centered at $A_0$. Let $\mathcal{B}_0(s)$ denote the 1-parameter family of horoballs centered at $A_0$ where 
$s$-parameter plays the role of the ``radius" of the horoball, the minimal Euclidean signed distance between the horoball and the center of the model $o$, taken negative if the horoball contains the model center.

\begin{lemma} [Local horoball density]
\label{lem:loc}
The local optimal horoball packing density of simply asymptotic Coxeter simplex $\cF_{\Gamma}$ is
\[
\delta_{opt}(\Gamma) = \frac{vol_{\mathbb{H}^3}(\mathcal{B}_0 \cap \cF_{\Gamma})}{vol_{\mathbb{H}^3}(\cF_{\Gamma})}.
\]
\end{lemma}

\begin{proof} The maximal horoball $\mathcal{B}_0(s)$ opposite $A_0$ with fundamental domain $\cF_{\Gamma}$ is tangent to the face of the simplex given by $\Bu_0$. This tangent point of $\mathcal{B}_0(s)$ and hyperface $\Bu_0$ is $[\mathbf{f}_0]$ the projection of vertex $A_0$ on plane $\Bu_0$ given by,
\begin{equation}
\mathbf{f}_0 =\ba_0 - \frac{\langle \ba_0, \mathbf{u}_0 \rangle}{\langle \Bu_0,\Bu_0 \rangle} 
\mathbf{u}_0.
\label{eq:u4fp}
\end{equation}

The value of the $s$-parameter for the maximal horoball can be read from the equation of the horosphere through $A_0$ and $\mathbf{f}_0$. 
The intersections $[\bh_i]$ of horosphere $\partial \mathcal{B}_0$ and the edges of the simplex $\cF_{\Gamma}$ are found by parameterizing the edges $\bh_i(\lambda) = \lambda \ba_0 +\ba_i$ $(1 \leq i \leq 3)$ then finding their intersections with $\partial \mathcal{B}_0$. 
The volume of the horospherical triangle determines the volume of the horoball piece by equation \eqref{eq:bolyai}.
The data for the horospherical triangle is obtained by finding hyperbolic distances $l_{ij}$ via equation \eqref{prop_dist},
$l_{ij} = d(H_i, H_j)$ where $d(\bh_i,\bh_j)= \arccos\left(\frac{-\langle \bh_i, \bh_j 
\rangle}{\sqrt{\langle \bh_i, \bh_i \rangle \langle \bh_j, \bh_j \rangle}}\right)$.
Moreover, the horospherical distances $L_{ij}$ are found by formula \eqref{eq:horo_dist}.
The intrinsic geometry of a horosphere is Euclidean, so the 
Cayley-Menger determinant gives the Euclidean volume $\mathcal{A}$ of the horospheric $2$-simplex $\mathcal{A}$,

\begin{equation}
\mathcal{A}^2 = \frac{1}{16}
\begin{vmatrix}
 0 & 1 & 1 & 1 \\
 1 & 0 & L_{1,2}^2 & L_{1,3}^2 \\
 1 & L_{1,2}^2 & 0 & L_{2,3}^2  \\
 1 & L_{1,3}^2 & L_{2,3}^2 &  0
 \end{vmatrix}.
 \end{equation}

The volume of the horoball piece contained in the fundamental simplex is

\begin{equation}
vol_{\mathbb{H}^3}(\mathcal{B}_0 \cap \cF_{\Gamma}) = \frac{1}{2}\mathcal{A}.
\end{equation}

The locally optimal horoball packing density of Coxeter Simplex $\cF_{\Gamma}$ is

\begin{equation}
\delta_{opt}(\cF_{\Gamma}) = \frac{vol_{\mathbb{H}^3}(\mathcal{B}_0 \cap \cF_{\Gamma})}{vol_{\mathbb{H}^3}(\cF_{\Gamma})}.
\end{equation}
\end{proof}

\begin{lemma}
\label{lem:glob}
The optimal horoball packing density $\delta_{opt}(\Gamma)$ of tiling $\cT_{\Gamma}$ and the local horoball packings density $\delta_{opt}(\cF_{\Gamma})$ are equal.
\end{lemma}

\begin{proof}
The local construction the the proof of Lemma \ref{lem:loc} is preserved by the isometric actions of $g \in \Gamma$. The Coxeter group $\Gamma$ extends the optimal local horoball packing density from the fundamental domain $\cF_{\Gamma}$ to the entire tiling $\cT_{\Gamma}$ of $\overline{\mathbb{H}}^3$, that is
$\delta_{opt}(\Gamma) = \delta_{opt}(\cF_{\Gamma})  = \frac{vol_{\mathbb{H}^3}(\mathcal{B}_0 \cap \cF_{\Gamma})}{vol_{\mathbb{H}^3}(\cF_{\Gamma})}$.
\end{proof}

Next, Lemma \ref{lemma:szirmai} gives the relationship of the sum of volumes of two tangent horoball pieces 
centered at two ideal vertices of the fundamental domain as the horoball types are varied while remaining tangent.

% V2

\begin{lemma}[\cite{Sz12}]
\label{lemma:szirmai}
Let $\cC_1$ and $\cC_2$ be two congruent convex cones in $\overline{\mathbb{H}}^3$ with ideal vertices $\xi_1,\xi_2\in \partial \overline{\mathbb{H}}^3$ that share the geodesic $(\xi_1,\xi_2)$ on their boundaries.  
Suppose $B_1(x)$ and $B_2(x)$ are two horoballs centered at $\xi_1$ and $\xi_2$ respectively, tangent at point $I(x)\in (\xi_1,\xi_2)$.
Define the reference point $I(0)$ as the midpoint along $(\xi_1,\xi_2)$ with the volumes of the horoball sectors satisfying 
\[
V(0) = 2 \; vol_{\mathbb{H}^3}(B_1(0) \cap \cC_1) = 2 \; vol_{\mathbb{H}^3}(B_2(0) \cap \cC_2).
\]  
Define the sum of contributions
\[
V(x)=vol_{\mathbb{H}^3}(B_1(x)\cap \cC_1)+vol_{\mathbb{H}^3}(B_2(x)\cap \cC_2).  
\]
If $x$ denotes the hyperbolic displacement of the tangent point $I(x)$ from a reference $I(0)$, then
\[
V(x)\;=\;V(0)\frac{e^{2x}+e^{-2x}}{2} \;=\; V(0)\cosh(2x),
\]
and $V(x)$ is strictly convex and strictly increasing as $x\to\pm\infty$.
\end{lemma}

\begin{proof}[Idea of proof]
The two tangent horoball pieces are exchanged by reflecting across the midpoint of $(\xi_1, \xi_2)$. Varying the tangency point along this edge multiplies one volume by $e^{2x}$ and the other by $e^{-2x}$, keeping their product fixed.  
Adding the two contributions gives the $\cosh(2x)$ law.  
See our paper \cite{Sz12} for details.
\end{proof}

%%%%%%%%%%%%%%%%%%%%%%%%%%%%%%%%%%%%%%%%%%%
%%%%%%%%%%%%%%%%%%%%%%%%%%%%%%%%%%%%%%%%%%%
\section{The Optimal Packing Densities of the Koszul Simplex Tilings}
\label{s:densities}
%%%%%%%%%%%%%%%%%%%%%%%%%%%%%%%%%%%%%%%%%%%
%%%%%%%%%%%%%%%%%%%%%%%%%%%%%%%%%%%%%%%%%%%

In this section we determine the optimal horoball packing densities of the Koszul-type Coxeter simplex tilings in $\overline{\bbH}^3$. 
Table \ref{table:simplex_list} summarizes these tilings together with their optimal packing densities, and Fig.~\ref{fig:lattice_of_subgroups_H3} displays the lattice of subgroups corresponding to the different commensurability classes. 
Throughout, we use Witt symbols to denote the various groups $\Gamma$.

%%%%%%%%%%%%%%%%%%%%%%%%%%%%%%%%%%%%%%%%%%%%%%%%%%%

\begin{figure}
\[
\begindc{\commdiag}[200]
% Commensurable to [3,3,6]
\obj(6,13)[336]{Commensurable to $[3,3,6]$}
\obj(0,8)[bp]{$\overline{BP}_3$}
\obj(2,10)[bv]{$\overline{BV}_3$}
\obj(2,6)[dp]{$\overline{DP}_3^{(2)}$}
\obj(4,8)[dv]{$\overline{DV}_3^{(2)}$}
\obj(6,12)[vv]{$\overline{V}_3$}
\obj(6,5)[vv2]{$\widehat{VV}_3^{(2)}$}
\obj(8,9)[yy]{$\overline{Y}_3^{(2)}$}
\obj(12,10)[pp]{$\overline{P}_3$}
\obj(12,7)[zz]{$\overline{Z}_3^{(2)}$}
\obj(10,6)[vp]{$\overline{VP}_3^{(3)}$}
\obj(10,4)[pp2]{$\overline{PP}_3^{(4)}$}

\mor{bp}{bv}{$2$}[\atright, \solidline]
\mor{bv}{dv}{$2$}[\atright, \solidline]
\mor{bp}{dp}{$2$}[\atright, \solidline]
\mor{dv}{dp}{$2$}[\atright, \solidline]
\mor{vv}{dv}{$5$}[\atright, \solidline]
\mor{vv}{yy}{$4$}[\atright, \solidline]
\mor{yy}{vv2}{$5$}[\atright, \solidline]
\mor{dv}{vv2}{$4$}[\atright, \solidline]
\mor{yy}{vp}{$3$}[\atright, \solidline]
\mor{vp}{pp2}{$2$}[\atright, \solidline]
\mor{vv}{pp}{$2$}[\atleft, \solidline]
\mor{vv}{zz}{$6$}[\atleft, \solidline]
\mor{zz}{vp}{$2$}[\atright, \solidline]

%Commensurable to [4,4,3]
\obj(18,13)[344]{Commensurable to $[3,4,4]$}
\obj(16,12)[rr]{$\overline{R}_3$}
\obj(16,10)[oo]{$\overline{O}_3$}
\obj(16,8)[br]{$\widehat{BR}_3$}
\obj(19,10)[nn]{$\overline{N}_3^{(2)}$}
\obj(19,8)[mm]{$\overline{M}_3^{(3)}$}
\obj(19,6)[rr2]{$\widehat{RR}_3^{(4)}$}

\mor{rr}{oo}{$2$}[\atright, \solidline]
\mor{oo}{br}{$2$}[\atright, \solidline]
\mor{rr}{nn}{$3$}[\atleft, \solidline]
\mor{nn}{mm}{$2$}[\atright, \solidline]
\mor{mm}{rr2}{$2$}[\atright, \solidline]

%Commensurable to [5,3,6]
\obj(4,3)[546]{Commensurable to $[5,3,6]$}
\obj(4,2)[hv]{$\overline{HV}_3$}
\obj(4,0)[hp]{$\overline{HP}_3$}

\mor{hv}{hp}{$2$}[\atright, \solidline]

%Other Non arithmetic
\obj(9,1)[cv]{$\widehat{CR}_3^{(2)}$}
\obj(13,1)[av]{$\widehat{AV}_3^{(2)}$}
\obj(17,1)[bv]{$\widehat{BV}_3^{(2)}$}
\obj(21,1)[hv]{$\widehat{HV}_3^{(2)}$}

\enddc
\]
\caption{Subgroup relations of the 23 paracompact (Koszul simplex) Coxeter groups in $\overline{\bbH}^3$, 17 arithmetic and 6 nonarithmetic. The superscript indicates the number of ideal vertices.}
\label{fig:lattice_of_subgroups_H3}
\end{figure}

\subsection{Commensurability Class of $[3,3,6]$}

The eleven groups in this class include the symmetries of the totally asymptotic hyperbolic tetrahedron and cube. 
We express their optimal packing densities in terms of $\Lambda(\tfrac{\pi}{3})$. 

\begin{proof}[Proof of Theorem \ref{thm:336}]
We distinguish four cases based on the number of ideal vertices. 
The relevant data is collected in Tables \ref{table:data_336_main}--\ref{table:data_two}. 
\begin{enumerate}
\item Coxeter simplices $\cF_{\Gamma}$ with $\Gamma \in \{\overline{V}_3, \overline{P}_3, \overline{BV}_3, \overline{BP}_3\}$ have a single ideal vertex (see Fig.~\ref{fig:lattice_of_subgroups_H3}). 
The local optimal packing densities follow from Lemma \ref{lem:loc} and extend to $\bbH^3$ by Lemma \ref{lem:glob}. 
Coordinates for $\{A_i\}$, hyperplanes $\{\Bu_i\}$ opposite each $A_i$, $s$-parameters for optimal horoballs, and intersection point coordinates are listed in Tables \ref{table:data_336_main} and \ref{table:data_336_one}. 

\item Coxeter simplices $\cF_{\Gamma}$ with $\Gamma \in \{\overline{Y}_3, \overline{Z}_3, \overline{DV}_3, \overline{DP}_3, \widehat{VV}_3\}$ have two ideal vertices. 
We defer the details of this case to Theorem \ref{thm:other}, which considers the nonarithmetic cases; the data used for the computations is given in Table \ref{table:data_two}. 
See also \cite[\S 4.2, \S 4.4]{KSz} for earlier versions of the argument in cases $\overline{Y}_3$ (denoted $\mathcal O_{(3,6,3)}$) and $\overline{DV}_3$.

\item Case $\Gamma = \overline{VP}_3$ has three ideal vertices. 
Following the method of \cite[\S 3.2.2--3.2.3]{KSz18}, the necessary data is given in Table \ref{table:data_336_main}.

\item Case $\Gamma = \overline{PP}_3$ corresponds to the ideal regular simplex and is a main result of \cite[\S 4.2]{KSz}.
\end{enumerate}
\end{proof}

For related techniques in {\it domain doubling} cases of fundamental domains, see \cite[\S 3.2.4]{KSz14} and \cite[\S 3]{KSz14}. 
Fig. \ref{fig:336_structure} illustrates the optimal packing structures in terms of the ratio of the volume of each horoball within a fundamental simplex cell.

\begin{figure}
\[
\begindc{\commdiag}[200]
% Commensurable to [3,3,6]
\obj(4,6)[dp]{$(\frac{1}{5},\frac{4}{5})_{\overline{DP}_3^{(2)}}$}
\obj(4,8)[dv]{$(\frac{1}{5},\frac{4}{5})_{\overline{DV}_3^{(2)}}$}
\obj(8,12)[vv]{$(\frac{1}{1})_{\overline{V}_3}$}
\obj(8,9)[yy]{$(\frac{1}{4},\frac{3}{4})_{\overline{Y}_3^{(2)}}$}
\obj(13,10)[pp]{$(\frac{1}{1})_{\overline{P}_3}$}
\obj(12,8)[zz]{$(\frac{1}{2},\frac{1}{2})_{\overline{Z}_3}$}
\obj(8,6)[vp]{$(\frac{1}{12},\frac{1}{6}, \frac{3}{4})_{\overline{VP}_3}$}
\obj(12,6)[vp2]{$(\frac{1}{4},\frac{1}{4}, \frac{1}{2})_{\overline{VP}_3}$}
\obj(8,4)[pp2]{$(\frac{1}{12},\frac{1}{12}, \frac{1}{12}, \frac{3}{4})_{\overline{PP}_3}$}
\obj(12,4)[pp3]{$(\frac{1}{4},\frac{1}{4}, \frac{1}{4}, \frac{1}{4})_{\overline{PP}_3}$}

\mor{dv}{dp}{$2$}[\atright, \solidline]
\mor{vv}{dv}{$5$}[\atright, \solidline]
\mor{vv}{yy}{$4$}[\atright, \solidline]
\mor{yy}{vp}{$3$}[\atright, \solidline]
\mor{vp}{pp2}{$2$}[\atright, \solidline]
\mor{vp2}{pp3}{$2$}[\atright, \solidline]
\mor{vv}{pp}{$2$}[\atleft, \solidline]
\mor{vv}{zz}{$6$}[\atleft, \solidline]
\mor{zz}{vp2}{$2$}[\atright, \solidline]

\enddc
\]
\caption{Classification of the optimal horoball packings in the commensurability class $[3,3,6]$, organized by the ratio of the contributions of the individual horoballs as fractions of $\Theta$.}
\label{fig:336_structure}
\end{figure}

\subsection{Commensurability Class of $[3,4,4]$}

This commensurability class includes the symmetries of the hyperbolic ideal regular octahedron, horoball packings of which were studied in detail in \cite[\S 4.3]{KSz}. 

\begin{proof}[Proof of Theorem \ref{thm:344}]
The proof follows the same methods as in Theorem \ref{thm:336}. 
The relevant data is collected in Tables \ref{table:data_344_one} and \ref{table:data_nml}. 
Note that the groups $\overline{N}_3$, $\overline{M}_3$, and $\widehat{RR}_3$, which have multiple ideal vertices, form a domain doubling sequence \cite{KSz14, EGM}. 
\end{proof}

For earlier partial results, see \cite[\S 4.3]{KSz}. 
Fig. \ref{fig:lattice_of_subgroups_344} illustrates the optimal packing structures in terms of the volume ratios of horoballs in a fundamental simplex cell. 

\begin{figure}
\[
\begindc{\commdiag}[200]
%Commensurable to [4,4,3]
\obj(0,12)[rr]{$(\frac{1}{1})_{\overline{R}_3}$}
\obj(0,10)[oo]{$(\frac{1}{1})_{\overline{O}_3}$}
\obj(0,8)[br]{$(\frac{1}{1})_{\widehat{BR}_3}$}
\obj(3,10)[nn]{$(\frac{2}{3},\frac{1}{3})_{\overline{N}_3}$}
\obj(3,8)[mm]{$(\frac{2}{3},\frac{1}{6}, \frac{1}{6})_{\overline{M}_3}$}
\obj(3,6)[rr2]{$(\frac{2}{3},\frac{1}{6}, \frac{1}{12}, \frac{1}{12})_{\widehat{RR}_3}$}

\mor{rr}{oo}{$2$}[\atright, \solidline]
\mor{oo}{br}{$2$}[\atright, \solidline]
\mor{rr}{nn}{$3$}[\atleft, \solidline]
\mor{nn}{mm}{$2$}[\atright, \solidline]
\mor{mm}{rr2}{$2$}[\atright, \solidline]
\enddc
\]
\caption{Classification of the optimal horoball packings in the commensurability class $[3,4,4]$, organized by the ratio of the contributions of the individual horoballs as fractions of $\Theta$.}
\label{fig:lattice_of_subgroups_344}
\end{figure}

\subsection{Nonarithmetic Cases}

There are six nonarithmetic asymptotic Coxeter simplices in five commensurability classes; see Table \ref{table:simplex_list}.  

\subsubsection{The commensurability class of $[5,3,6]$}
This commensurability class contains the symmetries of the asymptotic hyperbolic dodecahedron, cf.~\cite[\S 5]{KSz}. 
Both Coxeter simplices in this class have a single ideal vertex, see Table \ref{table:data_536}. 
Their volumes can be expressed in terms of the Lobachevsky function using orthoscheme decompositions following \cite{JKRT}: 
\[
vol_{\bbH^3}(\overline{HV}_3)  = \tfrac{1}{2}\Lambda\!\left(\tfrac{\pi}{3}\right)
+ \tfrac{1}{4}\Lambda\!\left(\tfrac{\pi}{3}\right)
- \tfrac{1}{4}\Lambda\!\left(\tfrac{\pi}{30}\right), 
\]
and 
\[
vol_{\bbH^3}(\overline{HP}_3) = 2 \cdot vol_{\bbH^3}(\overline{HV}_3).
\]

\begin{proof}[Proof of Theorem \ref{thm:536}]
Both Coxeter simplices $\cF_{\Gamma}$ with $\Gamma \in \{\overline{HV}_3, \overline{HP}_3\}$ have a single ideal vertex. 
The local optimal packing densities follow from Lemma \ref{lem:loc} and extend to the entire space by Lemma \ref{lem:glob}. 
Coordinates for the vertices $\{A_i\}$, hyperplanes $\{\Bu_i\}$ opposite each $A_i$, optimal horoball $s$-parameters, horoball intersection points, and closed-form expressions for the volumes of the horoball intersection pieces with $\cF_{\Gamma}$ are summarized in Table \ref{table:data_536}. 

In particular, Table \ref{table:data_536} implements the computation prescribed by Lemma \ref{lem:loc} step-by-step in coordinates:
\begin{enumerate}[(i)]
 \item compute the pole vectors $\mathbf{u}_i$ of the facets and the projection $\mathbf{f}_0$ of the ideal vertex via \eqref{eq:u4fp};
 \item determine the horosphere equation and read off the optimal $s$–parameter for the maximal admissible horoball;
 \item parametrize the simplex edges and find their intersections $H_i$ with the horosphere;
 \item compute hyperbolic edge distances and convert to horospherical chord lengths $L_{ij}$ using \eqref{prop_dist} and \eqref{eq:horo_dist};
 \item apply the Cayley–Menger determinant to obtain the Euclidean area $\mathcal{A}$ of the horospherical triangle;
 \item obtain the horoball piece volume $vol_{\bbH^3}(\mathcal{B}_0\cap\cF_{\Gamma})=\tfrac12\mathcal{A}$ (cf.~\eqref{eq:bolyai}), and record this contribution for the density formula \eqref{eq:density}.
\end{enumerate}
Thus the table provides a stepwise instantiation of Lemma \ref{lem:loc}, with the row $vol_{\bbH^3}(\mathcal{B}_0\cap\cF_{\Gamma})$ giving the closed-form contributions used to compute the packing density.
\end{proof}

\subsubsection{Commensurability classes of single simplices}
The nonarithmetic Coxeter simplices with cyclic diagrams decompose into orthoschemes, and their volumes are computed using the orthoscheme volume formulas from \cite{JKRT}. 
The decompositions are given below, where $\varphi = \frac{1+\sqrt{5}}{2}$ is the golden ratio: 

\begin{equation}
\vol(\widehat{AV}_3) = \vol([3,3,6]) + \vol([3,4,4]) + \vol([4,4,3]) + \vol([3,6,3]),
\end{equation}

\begin{multline}
\vol(\widehat{BV}_3) = \vol([4,3,6]) + \vol\!\left([4,\arctan\sqrt{2},\arccot\sqrt{2}]\right) \\
+ \vol\!\left([\arccot\tfrac{1}{\sqrt{2}},\tfrac{\pi}{2}-\arctan\sqrt{2},\tfrac{\pi}{3}]\right) + \vol([3,6,3]),
\end{multline}

\begin{multline}
\vol(\widehat{HV}_3) = \vol([5,3,6]) 
+ \vol\!\left([\tfrac{\pi}{5},\arctan \varphi,\arccot\varphi]\right)  \\
+ \vol\!\left([\arccot\varphi^{-1},\tfrac{\pi}{2}-\arctan \varphi,\tfrac{\pi}{3}]\right)  + \vol([3,6,3]),
\end{multline}

\begin{multline}
\vol(\widehat{CR}_3) = \vol([3,4,4]) + \vol([4,4,4]) \\
+ \vol\!\left([\tfrac{\pi}{3},\arctan\tfrac{1}{\sqrt{2}},\arccot\tfrac{1}{\sqrt{2}}]\right)  
+ \vol\!\left([\arccot\sqrt{2},\tfrac{\pi}{2}-\arctan\tfrac{1}{\sqrt{2}},\tfrac{\pi}{4}]\right).
\end{multline}

Numerical values for these volumes are listed in Table \ref{table:simplex_list}. 

\begin{proof}[Proof of Theorem \ref{thm:other}]
Each nonarithmetic simplex in Table \ref{table:other} has two ideal vertices; the relevant horoball parameters, volumes, and densities are collected there.  
We illustrate the argument with the case $\Gamma = \widehat{AV}_3$. 

From Table \ref{table:other}, the simplex $\cF_{\widehat{AV}_3}$ has two ideal vertices $A_0=(1,0,0,1)$ and $A_1=(1,0,-1,0)$.  
Let $B_0(\arctanh s_0)$ and $B_1(\arctanh s_1)$ denote horoballs centered at $A_0$ and $A_1$ with $s$-parameters $s_0,s_1$ as in Lemma \ref{lemma:s-parameter}.  
The horosphere $\partial B_1$ is obtained from $\partial B_0$ by the rotation $\text{Rot}_{A_1A_0} \in \mathrm{PGL}(4,\bbR)$ sending $A_1$ to $A_0$, represented by the matrix
\[
\text{Rot}_{A_1A_0} = 
\begin{pmatrix}
1 & 0 & 0 & 0  \\
0 & 1 & 0 & 0  \\
0 & 0 & 0 & -1 \\
0 & 0 & -1 & 0  
\end{pmatrix}.
\]

Let $x_i = \arctanh s_i = \beta(S_i,o,A_i)$ denote the hyperbolic distance from the model center $o=A_1=(1,0,\dots,0)$ to $S_i=(1,0,\dots,0,s_i)$ (with the appropriate rotation applied in the case of $A_1$).  
From Table \ref{table:other}, if $B_0$ is maximal then $s_0=0$, which forces $s_1=\tfrac{3}{5}$; if $B_1$ is maximal then $s_1=\tfrac{1}{13}(3-4\sqrt{3})$, which forces $s_0=\tfrac{1}{13}(4\sqrt{3}+3)$.  
The corresponding maximal horoballs $B_0(\arctanh 0)$ and $B_1(\arctanh \tfrac{1}{13}(3-4\sqrt{3}))$ are tangent to the hyperfaces $[\Bu_0]$ and $[\Bu_1]$ (see Table \ref{table:other}), and tangent to each other at $H_1$.  

Table \ref{table:other} also records the volumes of the horoball pieces:  
\[
V_{\max \cB_0}=\tfrac{1}{24}(2\sqrt{3}+3), \qquad V_{\cB_1}=\tfrac{1}{16\sqrt{3}}.
\]
By Lemma \ref{lem:loc} applied to each horoball, and Lemma \ref{lem:glob}, the optimal packing density is
\[
\delta_{opt}(\Gamma) 
= \frac{\tfrac{1}{16 \sqrt{3}} + \tfrac{1}{24} (2\sqrt{3}+3)}{\vol(\widehat{AV}_3)} 
\approx 0.838825,
\]
as recorded in the final row of Table \ref{table:other}.  
Finally, Lemma \ref{lemma:szirmai} implies that this value is indeed optimal.  

The same argument applies, mutatis mutandis, to the other three nonarithmetic simplices listed in Table \ref{table:other}, yielding the remaining values of $\delta_{opt}$.
\end{proof}

%%%%%%%%%%%%%%%%%%%%%%%%%%%%%%%%%%%%%%%%%%%
%%%%%%%%%%%%%%%%%%%%%%%%%%%%%%%%%%%%%%%%%%%
\newpage
\appendix
\section{Tables}
%%%%%%%%%%%%%%%%%%%%%%%%%%%%%%%%%%%%%%%%%%%
%%%%%%%%%%%%%%%%%%%%%%%%%%%%%%%%%%%%%%%%%%%

\begin{table}
    \begin{tabular}{cc|c|c|c}
    \hline
    Coxeter & ~ & Witt & Simplex & Packing\\
    Diagram  & Notation & Symbol & Volume &Density\\
    \hline
        Commensurable to $[3,3,6]$ & ~ & ~ & ~ & ~\\
    \hline
  \begin{tikzpicture}
	
	\draw (0.05,0) -- (1.5,0);

	\draw (0,0) circle (.05);
	\draw[fill=black] (.5,0) circle (.05); 
	\draw[fill=black] (1,0) circle (.05);
	\draw[fill=black] (1.5,0) circle (.05);

	\node at (1.25,0.175) {$6$};

\end{tikzpicture}
  & $[3,3,6]$ & $\overline{V}_3$ & $\frac{1}{8}\Lambda\left(\frac{1}{3} \pi \right)$ & $\Theta = \frac{1}{2 \sqrt{3} \Lambda\left(\frac{\pi }{3}\right)}$  \\
      \begin{tikzpicture}
	
	\draw (0.05,0) -- (1.45,0);

	\draw (0,0) circle (.05);
	\draw[fill=black] (.5,0) circle (.05); 
	\draw[fill=black] (1,0) circle (.05);
	\draw (1.5,0) circle (.05);

	\node at (.75,0.175) {$6$};

\end{tikzpicture}
  & $[3,6,3]$ & $\overline{Y}_3$ & $\frac{1}{2}\Lambda\left(\frac{1}{3} \pi \right)$ & $\Theta \approx 0.853276$  \\
      \begin{tikzpicture}
	
	\draw (0.05,0) -- (.5,0);
	\draw (.5,0) -- (0.97,.225);
	\draw (.5,0) -- (0.97,-.225);
	\draw (1,.2) -- (1,-.2);

	\draw (0,0) circle (.05);
	\draw[fill=black] (.5,0) circle (.05); 
	\draw (1,.25) circle (.05);
	\draw (1,-.25) circle (.05);

	\node at (.25,0.175) {$6$};

\end{tikzpicture}
  & $[6,3^{[3]}]$ & $\overline{VP}_3$ & $\frac{3}{2}\Lambda\left(\frac{1}{3} \pi \right)$ & $\Theta$  \\
\begin{tikzpicture}
	
	\draw[shorten <=0.05cm, shorten >=0.05cm] (0,0) -- (-.25,-.2);
	\draw[shorten <=0.05cm, shorten >=0.05cm] (0,0) -- (.25,-.2);
	\draw[shorten <=0.05cm, shorten >=0.05cm] (0,0) -- (0,.35);
	\draw[shorten <=0.05cm, shorten >=0.05cm] (-.25,-.2) -- (0,.35);
	\draw[shorten <=0.05cm, shorten >=0.05cm] (.25,-.2) -- (0,.35);
	\draw[shorten <=0.05cm, shorten >=0.05cm] (-.25,-.2) -- (.25,-.2);
	
	\draw (0,0) circle (.05);
	\draw (0,.35) circle (.05); 
	\draw (-.25,-.2) circle (.05);
	\draw (.25,-.2) circle (.05);
\end{tikzpicture}
  & $[3^{[3,3]}]$ & $\widehat{PP}_3$ & $3\Lambda\left(\frac{1}{3} \pi \right)$ & $\Theta$  \\
      \begin{tikzpicture}
	
	\draw (0.05,0) -- (.5,0);
	\draw (.5,0) -- (1,.25);
	\draw (.5,0) -- (1,-.25);
	\draw (1,.25) -- (1,-.25);

	\draw (0,0) circle (.05);
	\draw[fill=black] (.5,0) circle (.05); 
	\draw[fill=black] (1,.25) circle (.05);
	\draw[fill=black] (1,-.25) circle (.05);

\end{tikzpicture}
  & $[3,3^{[3]}]$ & $\overline{P}_3$ & $\frac{1}{4}\Lambda\left(\frac{1}{3} \pi \right)$ & $\Theta$  \\
   \begin{tikzpicture}
	
	\draw (0.05,0) -- (.45,0);
	\draw (0.55,0) -- (1.5,0);

	\draw (0,0) circle (.05);
	\draw (.5,0) circle (.05); 
	\draw[fill=black] (1,0) circle (.05);
	\draw[fill=black] (1.5,0) circle (.05);

	\node at (1.25,0.175) {$6$};
	\node at (.25,0.175) {$6$};
	
\end{tikzpicture}
  & $[6,3,6]$ & $\overline{Z}_3$ & $\frac{3}{4}\Lambda\left(\frac{1}{3} \pi \right)$ & $\Theta$  \\
     \begin{tikzpicture}
	
	\draw (0.05,0) -- (1.5,0);

	\draw (0,0) circle (.05);
	\draw[fill=black] (.5,0) circle (.05); 
	\draw[fill=black] (1,0) circle (.05);
	\draw[fill=black] (1.5,0) circle (.05);

	\node at (.25,0.175) {$4$};
	\node at (1.25,0.175) {$6$};
	
\end{tikzpicture}
  & $[4,3,6]$ & $\overline{BV}_3$ & $\frac{5}{16}\Lambda\left(\frac{1}{3} \pi \right)$ & $\rho =  \frac{2}{5 \sqrt{3} \Lambda\left(\frac{\pi }{3}\right)}$  \\
      \begin{tikzpicture}
	
	\draw (0.05,0) -- (.5,0);
	\draw (.5,0) -- (1,.25);
	\draw (.5,0) -- (1,-.25);
	\draw (1,.25) -- (1,-.25);

	\draw (0,0) circle (.05);
	\draw[fill=black] (.5,0) circle (.05); 
	\draw[fill=black] (1,.25) circle (.05);
	\draw[fill=black] (1,-.25) circle (.05);
	
	\node at (.25,0.175) {$4$};

\end{tikzpicture}
  & $[4,3^{[3]}]$ & $\overline{BP}_3$ & $\frac{5}{8}\Lambda\left(\frac{1}{3} \pi \right)$ & $\rho \approx  0.682620$ \\
      \begin{tikzpicture}
	
	\draw (0,0) -- (.5,0);
	\draw[shorten <=0.05cm, shorten >=0.05cm]  (.5,0) -- (1,.25);
	\draw[shorten <=0.05cm, shorten >=0.05cm]  (.5,0) -- (1,-.25);

	\draw[fill=black] (0,0) circle (.05);
	\draw[fill=black] (.5,0) circle (.05); 
	\draw (1,.25) circle (.05);
	\draw (1,-.25) circle (.05);

	\node at (.25,0.175) {$6$};

\end{tikzpicture}
  & $[4,3^{[1,1]}]$ & $\overline{DV}_3$ & $\frac{5}{8}\Lambda\left(\frac{1}{3} \pi \right)$ & $\Theta$  \\
      \begin{tikzpicture}
	
	\draw[shorten <=0.05cm, shorten >=0.05cm]  (0,-0.25) -- (0,0.25);
	\draw[shorten <=0.05cm, shorten >=0.05cm]  (.5,0) -- (0,.25);
	\draw[shorten <=0.05cm, shorten >=0.05cm]  (.5,0) -- (0,-.25);
	\draw[shorten <=0.05cm, shorten >=0.05cm]  (-.5,0) -- (0,.25);
	\draw[shorten <=0.05cm, shorten >=0.05cm]  (-.5,0) -- (0,-.25);
	
	\draw[fill=black] (0,0.25) circle (.05);
	\draw[fill=black] (0,-0.25) circle (.05); 
	\draw (-.5,0) circle (.05);
	\draw (.5,0) circle (.05);

\
\end{tikzpicture}
  & $[3^{[~]\times[~]}]$ & $\overline{DP}_3$ & $\frac{5}{4}\Lambda\left(\frac{1}{3} \pi \right)$ & $\Theta$  \\
          \begin{tikzpicture}
	
	\draw[shorten <=0.05cm, shorten >=0.05cm]  (0,0) -- (.5,0);
	\draw[shorten <=0.05cm, shorten >=0.05cm]  (0,0) -- (0,.5);
	\draw[shorten <=0.05cm, shorten >=0.05cm]  (0,.5) -- (.5,.5);
	\draw[shorten <=0.05cm, shorten >=0.05cm]  (.5,0) -- (.5,.5);

	\draw (0,0) circle (.05);
	\draw[fill=black] (.5,0) circle (.05); 
	\draw[fill=black] (0,.5) circle (.05);
	\draw (.5,.5) circle (.05);

	\node at (-.25,0.2) {$6$};
	\node at (.75,0.2) {$6$};

\end{tikzpicture}
  & $[(3,6)^{[2]}]$ & $\widehat{VV}_3$ & $\frac{5}{2}\Lambda\left(\frac{1}{3} \pi \right)$  & $\rho$  \\
          \hline
        Commensurable to $[3,4,4]$ & ~ & ~ & ~ & ~\\
    \hline
        \begin{tikzpicture}
	
	\draw (0.05,0) -- (1.5,0);

	\draw (0,0) circle (.05);
	\draw[fill=black] (.5,0) circle (.05); 
	\draw[fill=black] (1,0) circle (.05);
	\draw[fill=black] (1.5,0) circle (.05);

	\node at (1.25,0.175) {$4$};
	\node at (.75,0.175) {$4$};

\end{tikzpicture}
  & $[3,4,4]$ & $\overline{R}_3$ & $\frac{1}{6}\Lambda\left(\frac{1}{4} \pi \right)$ & $\sigma = \frac{3}{4 C}$ \\
      
      \begin{tikzpicture}
	
	\draw (0.05,0) -- (.5,0);
	\draw (.5,0) -- (1,.25);
	\draw (.5,0) -- (1,-.25);

	\draw (0,0) circle (.05);
	\draw[fill=black] (.5,0) circle (.05); 
	\draw[fill=black] (1,.25) circle (.05);
	\draw[fill=black] (1,-.25) circle (.05);

	\node at (.75,0.275) {$4$};
	\node at (.75,-0.275) {$4$};

\end{tikzpicture}
  & $[3,4^{1,1}]$ & $\overline{O}_3$ & $\frac{1}{3}\Lambda\left(\frac{1}{4} \pi \right)$ & $\sigma \approx  0.818808$  \\
           \begin{tikzpicture}
	
	\draw (0,0) -- (.5,0);
	\draw (0,0) -- (0,.475);
	\draw (0.05,.5) -- (.5,.5);
	\draw (.5,0) -- (.5,.5);

	\draw[fill=black] (0,0) circle (.05);
	\draw[fill=black] (.5,0) circle (.05); 
	\draw (0,.5) circle (.05);
	\draw[fill=black] (.5,.5) circle (.05);

	\node at (0.25,0.175) {$4$};
	\node at (.75,0.2) {$4$};

\end{tikzpicture}
  & $[(3^2,4^2)]$ & $\widehat{BR}_3$ & $\frac{2}{3}\Lambda\left(\frac{1}{4} \pi \right)$  & $\sigma$  \\
        \begin{tikzpicture}
	
	\draw (0.05,0) -- (1.45,0);

	\draw (0,0) circle (.05);
	\draw[fill=black] (.5,0) circle (.05); 
	\draw[fill=black] (1,0) circle (.05);
	\draw (1.5,0) circle (.05);

	\node at (1.25,0.175) {$4$};
	\node at (.75,0.175) {$4$};
	\node at (.25,0.175) {$4$};

\end{tikzpicture}
  & $[4,4,4]$ & $\overline{N}_3$ & $\frac{1}{2}\Lambda\left(\frac{1}{4} \pi \right)$ & $\sigma$ \\
      \begin{tikzpicture}
	
	\draw (.05,0) -- (.5,0);
	\draw (.5,0) -- (.95,.245);
	\draw (.5,0) -- (.95,-.245);

	\draw (0,0) circle (.05);
	\draw[fill=black] (.5,0) circle (.05); 
	\draw (1,.25) circle (.05);
	\draw (1,-.25) circle (.05);

	\node at (.75,0.275) {$4$};
	\node at (.75,-0.275) {$4$};
	\node at (.25,0.175) {$4$};

\end{tikzpicture}
  & $[4^{1,1,1}]$ & $\overline{M}_3$ & $\Lambda\left(\frac{1}{4} \pi \right)$ & $\sigma$ \\
        \begin{tikzpicture}
	
	\draw (.05,0) -- (.45,0);
	\draw (0,.05) -- (0,.45);
	\draw (.05,.5) -- (.45,.5);
	\draw (.5,.05) -- (.5,.45);

	\draw (0,0) circle (.05);
	\draw (.5,0) circle (.05); 
	\draw (0,.5) circle (.05);
	\draw (.5,.5) circle (.05);

	\node at (-.25,0.2) {$4$};
	\node at (.75,0.2) {$4$};
	\node at (.25,0.175) {$4$};
	\node at (.25,0.775) {$4$};

\end{tikzpicture}
  & $[4^{[4]}]$ & $\widehat{RR}_3$ & $2 \Lambda\left(\frac{1}{4} \pi \right)$ & $\sigma$ \\
    \hline
        Nonarithmetic & ~ & ~ & ~ & ~\\
        Commensurable to $[5,3,6]$ & ~ & ~ & ~ & ~\\
    \hline
        \begin{tikzpicture}
	
	\draw (.05,0) -- (1.5,0);

	\draw (0,0) circle (.05);
	\draw[fill=black] (.5,0) circle (.05); 
	\draw[fill=black] (1,0) circle (.05);
	\draw[fill=black] (1.5,0) circle (.05);

	\node at (.25,0.175) {$5$};	
	\node at (1.25,0.175) {$6$};

\end{tikzpicture}
  & $[5,3,6]$ & $\overline{HV}_3$ & $0.17150\dots$ & $ 0.550841\dots$  \\
      \begin{tikzpicture}
	
	\draw (0.05,0) -- (.5,0);
	\draw (.5,0) -- (1,.25);
	\draw (.5,0) -- (1,-.25);
	\draw (1,.25) -- (1,-.25);

	\draw (0,0) circle (.05);
	\draw[fill=black] (.5,0) circle (.05); 
	\draw[fill=black] (1,.25) circle (.05);
	\draw[fill=black] (1,-.25) circle (.05);

	\node at (.25,0.175) {$5$};

\end{tikzpicture}
  & $[5,3^{[3]}]$ & $\overline{HP}_3$ & $0.34300\dots$ & $0.550841\dots$  \\
    \hline
        Other Nonarithmetic & ~ & ~ & ~ & ~\\
    \hline
        \begin{tikzpicture}
	
	\draw (.05,0) -- (.45,0);
	\draw (0,.05) -- (0,.45);
	\draw (.05,.5) -- (.5,.5);
	\draw (.5,0) -- (.5,.5);

	\draw (0,0) circle (.05);
	\draw[fill=black] (.5,0) circle (.05); 
	\draw (0,.5) circle (.05);
	\draw[fill=black] (.5,.5) circle (.05);

	\node at (-.25,0.2) {$~$};
	\node at (.75,0.2) {$6$};

\end{tikzpicture}
  & $[(3^3,6)]$ & $\widehat{AV}_3$ & $0.364107\dots$ & $0.838825\dots$  \\
          \begin{tikzpicture}
	
	\draw (.05,0) -- (.5,0);
	\draw (0,.05) -- (0,.45);
	\draw (.05,.5) -- (.5,.5);
	\draw (.5,0) -- (.5,.5);

	\draw (0,0) circle (.05);
	\draw[fill=black] (.5,0) circle (.05); 
	\draw (0,.5) circle (.05);
	\draw[fill=black] (.5,.5) circle (.05);

	\node at (-.25,0.2) {$4$};
	\node at (.75,0.2) {$6$};

\end{tikzpicture}
  & $[(3,4,3,6)]$ & $\widehat{BV}_3$ & $0.525840\dots$ & $0.747914\dots$  \\
          \begin{tikzpicture}
	
	\draw (.05,0) -- (.45,0);
	\draw (0,.05) -- (0,.5);
	\draw (0,.5) -- (.5,.5);
	\draw (.5,0.05) -- (.5,.5);

	\draw (0,0) circle (.05);
	\draw (.5,0) circle (.05); 
	\draw[fill=black] (0,.5) circle (.05);
	\draw[fill=black] (.5,.5) circle (.05);

	\node at (-.25,0.2) {$5$};
	\node at (.75,0.2) {$6$};

\end{tikzpicture}
  & $[(3,5,3,6)]$ & $\widehat{HV}_3$ & $0.672985\dots$ & $0.655381\dots$  \\
          \begin{tikzpicture}
	
	\draw (0,0) -- (.5,0);
	\draw (0,0) -- (0,.45);
	\draw (0.05,.5) -- (.45,.5);
	\draw (.5,0.05) -- (.5,.45);

	\draw[fill=black] (0,0) circle (.05);
	\draw[fill=black] (.5,0) circle (.05); 
	\draw (0,.5) circle (.05);
	\draw (.5,.5) circle (.05);

	\node at (-.25,0.2) {$4$};
	\node at (.25,0.175) {$4$};
	\node at (.75,0.2) {$4$};

\end{tikzpicture}
  & $[(3,4^3)]$ & $\widehat{CR}_3$ & $0.556282\dots$ & $0.767195\dots$  \\
  
    \hline
    \end{tabular}
    \caption{The 23 asymptotic Coxeter Simplices in $\overline{\bbH}^3$ with optimal packing densities, $\Lambda$ is the Lobachevsky function and $C$ is Catalan's constant. Empty circles in the Coxeter diagram are reflection planes opposite to an ideal vertex.}
    \label{table:simplex_list}
\end{table}

\begin{table}[h!]
	\begin{tabular}{l|l|l|l|l}
		 \hline
		 \multicolumn{5}{c}{{\bf Coxeter Simplex Tilings} }\\
		 \hline
		 Witt Symb. & $\overline{V}_3$ &  $\overline{Y}_3$ &  $\overline{VP}_3$ &  \\
		 \hline
		 \multicolumn{5}{c}{{\bf Vertices of Simplex} }\\
		 \hline
		 $A_0$ & $(1, 0, 0, 1)^*$ & $(1,0,0,1)^*$ & $(1,0,0,1)^*$ &  \\
		 $A_1$ & $(1, 0, 0, 0)$ & $(1,0,0,0)$ & $(1,0,0,0)$ & \\
		 $A_2$ & $(1,\frac{1}{2},0,0)$ & $(1,\frac{1}{2},0,0)$ & $(1,\frac{\sqrt{3}}{2},-\frac{1}{2},0)^*$ & \\
		 $A_3$ & $(1,\frac{1}{2},\frac{\sqrt{3}}{6},0)$ & $(1,\frac{1}{2},\frac{\sqrt{3}}{2},0)^*$ & $(1,\frac{\sqrt{3}}{2},\frac{1}{2},0)^*$ & \\
		 \hline
		 \multicolumn{5}{c}{{\bf The form $\mbox{\boldmath$u$}_i$ of sides opposite $A_i$ }}\\
		\hline
		 $\mbox{\boldmath$u$}_0$ & $(0, 0, 0, 1)$ & $(0, 0, 0, 1)$ & $(0, 0, 0, 1)$ & \\
		 $\mbox{\boldmath$u$}_1$ & $(1, -2, 0, -1)$ & $(1, -2, 0, -1)$ & $(1,-\frac{2 \sqrt{3}}{3},0,-1)$ & \\
		 $\mbox{\boldmath$u$}_2$ & $(0,-\frac{\sqrt{3}}{3},1,0)$ & $(0,-\sqrt{3},1,0)$ & $(0,-\frac{\sqrt{3}}{3},1,0)$ & \\
		 $\mbox{\boldmath$u$}_3$ & $(0, 0, 1, 0)$ & $(0, 0, 1, 0)$ & $(0,\frac{\sqrt{3}}{3},1,0)$ & \\
		 \hline
		 \multicolumn{5}{c}{{\bf Maximal horoball-type parameter $s_i$ for horoball $\mathcal{B}_i$ at $A_i$ }}\\
		\hline
		 $s_0$ & $0$ & $0$ & $0$ & \\
		 %$s_1$ & - & - & - & \\
		 $s_2$ & - & - & $\frac{1}{7}$ & \\		 
		 $s_3$ & - & $\frac{1}{7}$ &  $\frac{1}{7}$ & \\
		 \hline
		 \multicolumn{5}{c}{ Horoball Parameters}\\
		\hline
		 $\mathcal{B}_0 \rightarrow $ & $0$ & $(0, \frac{3}{5})$ & $(0, \frac{3}{5}, \frac{3}{5})$ & \\
		 %$\mathcal{B}_1 \rightarrow $ & - & - & - & \\
		 $\mathcal{B}_2 \rightarrow $ & - & - & $(\frac{1}{2}, \frac{1}{7}, \frac{11}{13})$ & \\
		 $\mathcal{B}_3 \rightarrow $ & - & $(\frac{1}{2}, \frac{1}{7})$ & $(\frac{1}{2}, \frac{11}{13}, \frac{1}{7})$ & \\
		\hline
		\multicolumn{5}{c}{Packing Ratios w.r.t. $\delta_{opt}$}\\
		\hline
		 $\mathcal{B}_0 \rightarrow $ & $\frac{1}{1}\Theta$ & $(\frac{3}{4}, \frac{1}{4})\Theta$ & $(\frac{1}{2}, \frac{1}{4}, \frac{1}{4})\Theta$ & \\
		 %$\mathcal{B}_1 \rightarrow $ & - & - & - & \\
		 $\mathcal{B}_2 \rightarrow $ & - & - & $(\frac{1}{6}, \frac{3}{4}, \frac{1}{12}) \Theta$ & \\
		 $\mathcal{B}_3 \rightarrow $ & - & $(\frac{1}{4}, \frac{3}{4})\Theta$ & $(\frac{1}{6}, \frac{1}{12},\frac{3}{4}) \Theta$ & \\
		\hline
		 \multicolumn{5}{c}{ {\bf Optimal Horoball Packing Density} }\\
		\hline
		$\delta_{opt} $ & $\Theta$ & $\Theta$ & $\Theta$ & \\
		\hline
	\end{tabular}
	\caption{Data for the multiply asymptotic Coxeter simplex tilings in the projective Cayley–Klein ball model of radius $1$, centered at $(1,0,0,0)$. An asterisk $^*$ indicates an ideal vertex.}
	\label{table:data_336_main}
\end{table}

\begin{table}[h!]
	\begin{tabular}{l|l|l|l}
		 \hline
		 \multicolumn{4}{c}{{\bf Coxeter Simplex Tilings} }\\
		\hline
		 Witt Symb. & $\overline{P}_3$ &  $\overline{BV}_3$ &  $\overline{BP}_3$ \\
		 \hline
		 \multicolumn{4}{c}{{\bf Vertices of Simplex} }\\
		 \hline
		 $A_0$ & $(1,0,0,1)^*$ & $(1,0,0,1)^*$ & $(1,0,0,1)^*$ \\
		 $A_1$ & $(1,0,0,0)$ & $(1,0,0,0)$ & $(1,0,0,0)$ \\
		 $A_2$ & $(1,\frac{1}{2},-\frac{\sqrt{3}}{6},0)$ & $(1,\frac{\sqrt{2}}{2},0,0)$ & $(1,\frac{\sqrt{2}}{2},-\frac{\sqrt{6}}{6},0)$ \\
		 $A_3$ & $(1,\frac{1}{2},\frac{\sqrt{3}}{6},0)$ & $(1,\frac{\sqrt{2}}{2},\frac{\sqrt{6}}{6},0)$ & $(1,\frac{\sqrt{2}}{2},\frac{\sqrt{6}}{6},0)$ \\
		 \hline
		 \multicolumn{4}{c}{{\bf The form $\mbox{\boldmath$u$}_i$ of sides opposite to $A_i$ }}\\
		\hline
		 $\mbox{\boldmath$u$}_0$ & $(0, 0, 0, -1)$ & $(0,0,0,1)$ & $(0, 0, 0, 1)$ \\
		 $\mbox{\boldmath$u$}_1$ & $(1, -2, 0, -1)$ & $(1,-\sqrt{2},0,-1)$ & $(1,-\sqrt{2},0,-1)$ \\
		 $\mbox{\boldmath$u$}_2$ & $(0,-\frac{\sqrt{3}}{3},1,0)$ & $(0,-\frac{\sqrt{12}}{6},1,0)$ & $(0,-\frac{\sqrt{12}}{6},1,0)$ \\
		 $\mbox{\boldmath$u$}_3$ & $(0,\frac{\sqrt{3}}{3},1,0)$ & $(0,0,1,0)$ & $(0,\frac{\sqrt{12}}{6},1,0)$ \\
		 \hline
		 \multicolumn{4}{c}{{\bf Maximal horoball parameter $s_0$ }}\\
		\hline
		 $s_0$ & $0$ & $0$ & $0$ \\
		\hline
		 \multicolumn{4}{c}{ {\bf Intersections $H_i = \mathcal{B}(A_0,s_0)\cap A_0A_i$ of horoballs with simplex edges}}\\
		\hline
		 $H_1$ & $(1, 0, 0, 0)$ & $(1, 0, 0, 0)$ & $(1, 0, 0, 0)$  \\
		 $H_2$ & $(1,\frac{3}{7},-\frac{\sqrt{3}}{7},\frac{1}{7})$ & $(1,\frac{2 \sqrt{2}}{5},0,\frac{1}{5})$ & $(1,\frac{3}{4 \sqrt{2}},-\frac{1}{4}\sqrt{\frac{3}{2}},\frac{1}{4})$ \\
		 $H_3$ & $(1,\frac{3}{7},\frac{\sqrt{3}}{7},\frac{1}{7})$ & $1,\frac{3}{4 \sqrt{2}},\frac{\sqrt{\frac{3}{2}}}{4},\frac{1}{4}$ & $(1,\frac{3}{4 \sqrt{2}},\frac{1}{4}\sqrt{\frac{3}{2}},\frac{1}{4})$ \\
		 \hline
		 \multicolumn{4}{c}{ {\bf Volume of maximal horoball piece }}\\
		\hline
		 $vol(\mathcal{B}_0 \cap \mathcal{F})$ & $\frac{1}{8 \sqrt{3}}$ & $\frac{1}{8 \sqrt{3}}$ & $\frac{1}{4 \sqrt{3}}$ \\
		\hline
		\multicolumn{4}{c}{ {\bf Optimal Packing Density}}\\
		\hline
		 $\delta_{opt}$ & $\Theta$ & $\rho$ & $\rho$ \\ 
		\hline
	\end{tabular}
	\caption{Data for the simply asymptotic Coxeter simplex tilings in the projective Cayley–Klein ball model of radius $1$, centered at $(1,0,0,0)$. An asterisk $^*$ indicates an ideal vertex.}
	\label{table:data_336_one}
\end{table}

\begin{landscape}

\begin{table}[h!]
\resizebox{\columnwidth}{!}{%
	\begin{tabular}{l|l|l|l|l}
		 \hline
		 \multicolumn{5}{c}{{\bf Coxeter Simplex Tilings} }\\
		 \multicolumn{5}{c}{{\bf Doubly Asymptotic Nonarithmetic} } \\
		 \hline
		 Witt Symb.: & $\overline{DV}_3$ & $\overline{DP}_3$ & $\overline{VV}_3$ & $\overline{Z}_3$\\
		 \hline
		 \multicolumn{4}{c}{{\bf Vertices of Simplex}}\\
		 \hline
		 $A_0$ & $(1, 0, 0, 1)^*$ & $(1, 0, 0, 1)^*$ & $(1, 0, 0, 1)^*$ & $(1, 0, 0, 1)^*$\\
		 $A_1$ & $(1, 0, 0, -1)^*$ & $(1, 0, 0, -1)^*$ & $(1, 0, 0, -1)^*$ & $(1, 0, 0, -1)^*$ \\
		 $A_2$ & $(1,\frac{\sqrt{6}}{3},0,0)$ & $(1,\frac{\sqrt{2}}{2},\frac{\sqrt{6}}{6},0)$ & $(1,\sqrt{4 \sqrt{3}-6},0,2-\sqrt{3})$ & $(1,\frac{\sqrt{3}}{3},-\frac{1}{3},\frac{1}{3})$\\
		 $A_3$ & $(1,\frac{\sqrt{6}}{4},\frac{\sqrt{2}}{4},0)$ & $(1,\frac{\sqrt{2}}{2},-\frac{\sqrt{6}}{6},0)$ & $(1,0,\sqrt{4 \sqrt{3}-6},\sqrt{3}-2)$ & $(1,\frac{\sqrt{3}}{3},\frac{1}{3},-\frac{1}{3})$\\
		 \hline
		 \multicolumn{5}{c}{{\bf The form $\mbox{\boldmath$u$}_i$ of sides opposite $A_i$ }}\\
		\hline
		 $\Bu_0$ & $(1,-\frac{\sqrt{6}}{2},-\frac{\sqrt{2}}{2},1)$ & $(1,-\sqrt{2},0,1)$ & $(1,-\sqrt[4]{3},-\frac{1}{\sqrt[4]{3}},1)$ & $(1,-\sqrt{3},1,1)$\\
		 $\Bu_1$ & $(1,-\frac{\sqrt{6}}{2},-\frac{\sqrt{2}}{2},-1)$ & $(1,-\sqrt{2},0,-1)$ & $(1,-\sqrt[4]{3},-\frac{1}{\sqrt[4]{3}},-1)$ & $(1,-\sqrt{3},-1,-1)$ \\
		 $\Bu_2$ & $(0,-\frac{\sqrt{3}}{3},1,0)$ & $(0,\frac{\sqrt{12}}{6},1,0)$ & $(0,\frac{\sqrt{12}}{6},1,0)$ & $(0,-\frac{\sqrt{3}}{3},1,0)$\\
		 $\Bu_3$ & $(0, 0, 1, 0)$ & $(0,-\frac{\sqrt{12}}{6},1,0)$ & $(0,-\frac{\sqrt{12}}{6},1,0)$ & $(0,\frac{\sqrt{3}}{3},1,0)$\\
		 \hline
		 \multicolumn{5}{c}{{\bf Maximal horoball-type parameter $s_i$ for horoball $\mathcal{B}_i$ at $A_i$ }}\\
		\hline
		 $\max s_0 \implies s_1 $ & $-\frac{1}{3} \implies \frac{1}{3}$ & $-\frac{1}{3} \implies \frac{1}{3}$ & $-2 + \sqrt{3} \implies 2 - \sqrt{3}$ & $0 \implies 0$ \\
		 $\max s_1 \implies s_0 $ & $-\frac{1}{3} \implies \frac{1}{3}$ & $-\frac{1}{3} \implies \frac{1}{3}$ & $-2 + \sqrt{3} \impliedby 2 - \sqrt{3}$ & $0 \impliedby 0$ \\
		 \hline
		 \multicolumn{5}{c}{ {\bf Volumes of optimal horoball pieces $V_i = vol(\mathcal{B}_i \cap \mathcal{F}_{\Gamma})$}}\\
		\hline
		 $V_{\max \mathcal{B}_0} \implies V_{\mathcal{B}_1}$ & $\frac{1}{4 \sqrt{3}} \implies \frac{1}{16 \sqrt{3}}$ & $\frac{1}{2 \sqrt{3}} \implies \frac{1}{8 \sqrt{3}}$ & $\frac{\sqrt{3}}{4} \implies \frac{1}{4 \sqrt{3}}$ & $ \frac{\sqrt{3}}{16} \implies \frac{\sqrt{3}}{16}$ \\
		 $V_{\mathcal{B}_0} \impliedby V_{\max \mathcal{B}_1}$ & $\frac{1}{4 \sqrt{3}} \impliedby \frac{1}{16 \sqrt{3}}$ & $\frac{1}{2 \sqrt{3}} \impliedby \frac{1}{8 \sqrt{3}}$ & $\frac{1}{4 \sqrt{3}} \impliedby \frac{\sqrt{3}}{4}$ & $ \frac{\sqrt{3}}{16} \implies \frac{\sqrt{3}}{16}$ \\
		\hline
		 \multicolumn{5}{c}{ {\bf Densities of horoball pieces $\delta_i = vol(\mathcal{B}_i \cap \mathcal{F}_{\Gamma})$}}\\
		\hline
		 $(\delta_{\max s_0},\delta_{s_1})$ & $(\frac{4}{5}, \frac{1}{5})\delta_{opt}$ & $(\frac{4}{5}, \frac{1}{5})\delta_{opt}$ & $(\frac{3}{4}, \frac{1}{4})\delta_{opt}$ & $(\frac{1}{2}, \frac{1}{2})\delta_{opt}$ \\
		 $(\delta_{s_0},\delta_{\max s_1})$ & $(\frac{1}{5}, \frac{4}{5})\delta_{opt}$ & $(\frac{4}{5}, \frac{4}{5})\delta_{opt}$ & $(\frac{1}{4}, \frac{3}{4})\delta_{opt}$ & $(\frac{1}{2}, \frac{1}{2})\delta_{opt}$ \\
		\hline
		\multicolumn{5}{c}{ {\bf Optimal Horoball Packing Density} }\\
		\hline
		$\delta_{opt}$ & $\Theta$ & $\Theta$ & $\rho$ & $\Theta$ \\
		\hline
	\end{tabular}%
}
	\caption{Data for doubly asymptotic Coxeter simplex tilings in the Cayley–Klein ball model of radius $1$, centered at $(1,0,0,0)$. Vertices marked with $^*$ denote ideal vertices.}
	\label{table:data_two}
\end{table}

\end{landscape}

\newpage

\begin{table}[h!]
	\begin{tabular}{l|l|l|l}
		\hline
		\multicolumn{4}{c}{{\bf Coxeter Simplex Tilings} }\\
		\hline
		 Witt Symb. & $\overline{R}_3$ &  $\overline{O}_3$ &  $\widehat{BR}_3$ \\
		 \hline
		 \multicolumn{4}{c}{{\bf Vertices of Simplex} }\\
		 \hline
		 $A_0$ & $(1, 0, 0, 1)^*$ & $(1, 0, 0, 1)^*$ & $(1, 0, 0, 1)^*$  \\
		 $A_1$ & $(1,0,0,0)$ & $(1,0,0,0)$ & $(1,0,0,0)$  \\
		 $A_2$ & $(1,0,\frac{\sqrt{3}}{4},\frac{1}{4})$ & $(1,\frac{1}{2},-\frac{1}{2},0)$ & $(1,0,\frac{\sqrt{2}}{2},\frac{1}{2})$  \\
		 $A_3$ & $(1,\frac{\sqrt{3}}{3},0,0)$ & $(1,\frac{1}{2},\frac{1}{2},0)$ & $(1,\frac{\sqrt{2}}{2},0,\frac{1}{2})$   \\
		 \hline
		 \multicolumn{4}{c}{{\bf The form $\mbox{\boldmath$u$}_i$ of sides opposite to $A_i$ }}\\
		\hline
		 $\mbox{\boldmath$u$}_0$ & $(0,0,-\frac{1}{\sqrt{3}},1)$ & $(0, 0, 0, 1)$ & $(0,-\frac{\sqrt{2}}{2},-\frac{\sqrt{2}}{2},1)$  \\
		 $\mbox{\boldmath$u$}_1$ & $(0, 0, 1, 0)$ & $(1, -2, 0, -1)$ & $(1,-\frac{\sqrt{2}}{2},-\frac{\sqrt{2}}{2},-1)$  \\
		 $\mbox{\boldmath$u$}_2$ & $(1,-\sqrt{3},-\sqrt{3},-1)$ & $(0,-1,1,0)$ & $(0,0,1,0)$    \\
		 $\mbox{\boldmath$u$}_3$ & $(0, 1, 0, 0)$ & $(0, 1, 1, 0)$ & $(0, 1, 0, 0)$    \\
		 \hline
		 \multicolumn{4}{c}{{\bf Maximal horoball parameter $s_0$ }}\\
		\hline
		 $s_0$ & $\frac{1}{7}$ & $0$ & $\frac{1}{3}$   \\
		\hline
		 \multicolumn{4}{c}{ {\bf Intersections $H_i = \mathcal{B}(A_0,s_0)\cap A_0A_i$ of horoballs with simplex edges}}\\
		\hline
		 $H_1$ & $(1, 0, 0, 0)$ & $(1, 0, 0, 0)$ & $(1, 0, 0, 0)$   \\
		 $H_2$ & $(1,0,\frac{2 \sqrt{3}}{7},\frac{1}{7})$ & $(1,\frac{2}{5},-\frac{2}{5},\frac{1}{5})$ & $(1,0,\frac{1}{\sqrt{2}},\frac{1}{2})$ \\
		 $H_3$ & $(1,\frac{2 \sqrt{3}}{7},0,\frac{1}{7})$ & $(1,\frac{2}{5},\frac{2}{5},\frac{1}{5})$ & $(1,\frac{1}{\sqrt{2}},0,\frac{1}{2})$  \\
		 \hline
		 \multicolumn{4}{c}{ {\bf Volume of maximal horoball piece }}\\
		\hline
		 $vol(\mathcal{B}_0 \cap \mathcal{F})$ & $\frac{1}{12}$ & $\frac{1}{8}$ & $\frac{1}{2}$  \\
		\hline
		\multicolumn{4}{c}{ {\bf Optimal Packing Density}}\\
		\hline
		 $\delta_{opt}$ & $\frac{3}{4 C}$ & $\frac{3}{4 C}$ & $\frac{3}{4 C}$   \\
	\end{tabular}
	\caption{Data for the simply asymptotic Coxeter simplex tilings in the projective Cayley–Klein ball model of radius $1$, centered at $(1,0,0,0)$. An asterisk $^*$ indicates an ideal vertex.}
	\label{table:data_344_one}
\end{table}

\begin{table}[h!]
\resizebox{\columnwidth}{!}{%
	\begin{tabular}{l|l|l|l}
		 \hline
		 \multicolumn{4}{c}{{\bf Coxeter Simplex Tilings} }\\
		 \hline
		 & \multicolumn{3}{|c}{{\bf Doubling Sequence} } \\
		 \hline
		 Witt Symb. & $\overline{N}_3$ &  $\overline{M}_3$ &  $\widehat{RR}_3$ \\
		 \hline
		 \multicolumn{4}{c}{{\bf Vertices of Simplex} }\\
		 \hline
		 $A_0$ & $(1, 0, 0, 1)^*$ & $(1, 0, 0, 1)^*$ & $(1, 0, 0, 1)^*$ \\
		 $A_1$ & $(1,\sqrt{\sqrt{2}-1},-\sqrt{5 \sqrt{2}-7},3-2 \sqrt{2})$ & $(1, 0, 0, 0)$ & $(1, 1, 0, 0)^*$ \\
		 $A_2$ & $(1,\sqrt{\sqrt{2}-1},\sqrt{5 \sqrt{2}-7},2 \sqrt{2}-3)$ & $(1,\frac{\sqrt{2}}{2},-\frac{\sqrt{2}}{2},0)$ & $(1,0,0,-1)^*$ \\
		 $A_3$ & $(1, 0, 0, -1)^*$ & $(1,\frac{\sqrt{2}}{2},\frac{\sqrt{2}}{2},0)$ & $(1, 0, 1, 0)^*$ \\
		 \hline
		 \multicolumn{4}{c}{{\bf The form $\mbox{\boldmath$u$}_i$ of sides opposite $A_i$ }}\\
		\hline
		 $\mbox{\boldmath$u$}_0$ & $(1,-\sqrt{\sqrt{2}+1},\sqrt{\sqrt{2}-1},1)$ & $(0, 0, 0, 1)$ & $(1,-1,-1,1)$ \\
		 $\mbox{\boldmath$u$}_1$ & $(0,1-\sqrt{2},1,0)$ & $(1,-\sqrt{2},0,-1)$ & $(0, 1, 0, 0)$ \\
		 $\mbox{\boldmath$u$}_2$ & $(0,\sqrt{2}-1,1,0)$ & $(0, -1, 1, 0)$ & $(1,-1,-1,-1)$ \\
		 $\mbox{\boldmath$u$}_3$ & $(1,-\sqrt{\sqrt{2}+1},-\sqrt{\sqrt{2}-1},-1)$ & $(0, 1, 1, 0)$ & $(0,0,1,0)$ \\
		 \hline
		 \multicolumn{4}{c}{{\bf Maximal horoball-type parameter $s_i$ for horoball $\mathcal{B}_i$ at $A_i$ }}\\
		\hline
		 $s_0$ & $\frac{\sqrt{2}-2}{\sqrt{2}+2}$ & 0 & $-\frac{1}{3}$ \\
		 $s_1$ & $\frac{\sqrt{2}-2}{\sqrt{2}+2}$ & 0 & $0$\\
		 $s_2$ & - & 0 & $-\frac{1}{3}$ \\
		 $s_3$ & - &  -& $0$\\
		 \hline
		 \multicolumn{4}{c}{ Horoball Parameters}\\
		\hline
		 $\mathcal{B}_0 \rightarrow \mathcal{B}_i$ & $s_0=\frac{\sqrt{2}-2}{\sqrt{2}+2} \rightarrow s_1=\frac{\sqrt{2}+1}{5 \sqrt{2}+7}$ & $s_0 = 0 \rightarrow (s_2 = \frac{3}{5}, s_3 = \frac{3}{5})$ & $s_0 = -\frac{1}{3} \rightarrow (s_1 = \frac{7}{9}, s_2 = \frac{1}{3}, s_3 = \frac{7}{9})$\\
		 $\mathcal{B}_i \rightarrow (\mathcal{B}_0, \mathcal{B}_2)$ & $s_1=\frac{\sqrt{2}-2}{\sqrt{2}+2} \rightarrow s_0=\frac{\sqrt{2}+1}{5 \sqrt{2}+7}$ & $s_1 = 0 \rightarrow (s_0 = \frac{3}{5}, s_3 = \frac{3}{5})$ &  - \\
		  $\mathcal{B}_2 \rightarrow (\mathcal{B}_0, \mathcal{B}_1)$ & - & $s_3 = 0 \rightarrow (s_0 = \frac{3}{5}, s_2 = \frac{3}{5})$ & $s_2 = -\frac{1}{3} \rightarrow (s_1 = \frac{7}{9}, s_0 = \frac{1}{3}, s_3 = \frac{7}{9})$ \\
		 $\mathcal{B}_1 \rightarrow (\mathcal{B}_0, \mathcal{B}_2, \mathcal{B}_4)$ & - & - & - \\
		\hline
		\multicolumn{4}{c}{Packing Ratios w.r.t. $\delta_{opt}$}\\
		\hline
		 $\mathcal{B}_0 \rightarrow \mathcal{B}_i$ & $(\frac{2}{3}, \frac{1}{3})\delta_{opt}$ & $(\frac{2}{3}, \frac{1}{6}, \frac{1}{6})\delta_{opt}$ & $(\frac{2}{3},\frac{1}{6}, \frac{1}{12}, \frac{1}{12})\delta_{opt}$ \\
		 $\mathcal{B}_i \rightarrow (\mathcal{B}_0, \mathcal{B}_2)$ & $(\frac{1}{3}, \frac{2}{3})\delta_{opt}$ & $(\frac{1}{6}, \frac{2}{3}, \frac{1}{6})\delta_{opt}$ & - \\
		 $\mathcal{B}_2 \rightarrow (\mathcal{B}_0, \mathcal{B}_1)$ & - & $(\frac{1}{6}, \frac{1}{6}, \frac{2}{3})\delta_{opt}$ & - \\
		 $\mathcal{B}_1 \rightarrow(\mathcal{B}_0, \mathcal{B}_2, \mathcal{B}_4)$ & - & - & - \\
		\hline
		 \multicolumn{4}{c}{ {\bf Optimal Horoball Packing Density} }\\
		\hline
		$\delta_{opt} $ & $\frac{3}{4C}$ & $\frac{3}{4C}$ & $\frac{3}{4C}$ \\
		\hline
	\end{tabular}%
}
	\caption{Data for multiply asymptotic Coxeter simplex tilings in the Cayley-Klein ball model of radius $1$ centered at $(1,0,0,0)$. An asterisk $^*$ indicates an ideal vertex.}
	\label{table:data_nml}
\end{table}

\newpage

\begin{table}[h!]
\resizebox{\columnwidth}{!}{%
	\begin{tabular}{l|l|l}
		 \hline
		 \multicolumn{3}{c}{{\bf Coxeter Simplex Tilings} }\\
		\hline
		 Witt Symb. & $\overline{HV}_3$ &  $\overline{HP}_3$  \\
		 \hline
		 \multicolumn{3}{c}{{\bf Vertices of Simplex} }\\
		 \hline
		 $A_0$ & $(1, 0, 0, 1)^*$ & $(1, 0, 0, 1)^*$ \\
		 $A_1$ & $(1, 0, 0, 0)$ & $(1, 0, 0, 0)$ \\
		 $A_2$ & $(1,\frac{1}{4} \left(\sqrt{5}+1\right),0,0)$ & $(1,\frac{1}{4} \left(\sqrt{5}+1\right),-\frac{1}{2} \sqrt{\frac{1}{6} \left(\sqrt{5}+3\right)},0)$ \\
		 $A_3$ & $(1,\frac{1}{4} \left(\sqrt{5}+1\right),\frac{1}{2} \sqrt{\frac{1}{6} \left(\sqrt{5}+3\right)},0)$ & $(1,\frac{1}{4} \left(\sqrt{5}+1\right),\frac{1}{2} \sqrt{\frac{1}{6} \left(\sqrt{5}+3\right)},0)$ \\
		 \hline
		 \multicolumn{3}{c}{{\bf The form $\mbox{\boldmath$u$}_i$ of sides opposite to $A_i$ }}\\
		\hline
		 $\mbox{\boldmath$u$}_0$ & $(0, 0, 0, 1)$ & $(0, 0, 0, 1)$ \\
		 $\mbox{\boldmath$u$}_1$ & $(1,1-\sqrt{5},0,-1)$ & $(1,1-\sqrt{5},0,-1)$ \\
		 $\mbox{\boldmath$u$}_2$ & $(0,-\frac{1}{\sqrt{3}},1,0)$ & $(0,-\frac{\sqrt{3}}{3},1,0)$ \\
		 $\mbox{\boldmath$u$}_3$ & $(0,0,1,0)$ & $(0,\frac{\sqrt{3}}{3},1,0)$ \\
		 \hline
		 \multicolumn{3}{c}{{\bf Maximal horoball parameter $s_0$ }}\\
		\hline
		 $s_0$ & $0$ & $0$ \\
		\hline
		 \multicolumn{3}{c}{ {\bf Intersections $H_i = \mathcal{B}(A_0,s_0)\cap A_0A_i$ of horoballs with simplex edges}}\\
		\hline
		 $H_1$ & $(1,0,0,0)$ & $(1, 0, 0, 0)$ \\
		 $H_2$ & $(1,\frac{2}{89} \left(9 \sqrt{5}+7\right),0,\frac{1}{89} \left(4 \sqrt{5}+13\right))$ & $(1,\frac{3}{110} \left(7 \sqrt{5}+5\right),-\frac{\sqrt{6 \left(\sqrt{5}+3\right)}}{\sqrt{5}+15},\frac{1}{55} \left(3 \sqrt{5}+10\right))$ \\
		 $H_3$ & $(1,\frac{3}{110} \left(7 \sqrt{5}+5\right),\frac{\sqrt{6 \left(\sqrt{5}+3\right)}}{\sqrt{5}+15},\frac{1}{55} \left(3 \sqrt{5}+10\right))$ & $(1,\frac{3}{110} \left(7 \sqrt{5}+5\right),\frac{\sqrt{6 \left(\sqrt{5}+3\right)}}{\sqrt{5}+15},\frac{1}{55} \left(3 \sqrt{5}+10\right) )$ \\
		 \hline
		 \multicolumn{3}{c}{ {\bf Volume of maximal horoball piece }}\\
		\hline
		 $vol(\mathcal{B}_0 \cap \mathcal{F})$ & $\frac{1}{16} \sqrt{\frac{1}{6} \left(3 \sqrt{5}+7\right)}$ & $\frac{1}{8} \sqrt{\frac{1}{6} \left(3 \sqrt{5}+7\right)}$ \\
		\hline
		\multicolumn{3}{c}{ {\bf Optimal Packing Density}}\\
		\hline
		 $\delta_{opt}$ & $0.550841$ & $0.550841$ \\
		\hline
	\end{tabular}%
}
	\caption{Data for the simply asymptotic Coxeter simplex tilings in the projective Cayley-Klein ball model of radius 1 centered at (1,0,0,0). An asterisk $^*$ indicates an ideal vertex.}
	\label{table:data_536}
\end{table}

\begin{landscape}

\begin{table}[h!]
\resizebox{\columnwidth}{!}{%
	\begin{tabular}{l|l|l|l|l}
		 \hline
		 \multicolumn{5}{c}{{\bf Coxeter Simplex Tilings} }\\
		 \multicolumn{5}{c}{{\bf Doubly Asymptotic Nonarithmetic} } \\
		 \hline
		 Witt Symb.: & $\widehat{CR}_3$ & $\widehat{AV}_3$ & $\widehat{BV}_3$ & $\widehat{HV}_3$\\
		 \hline
		 \multicolumn{4}{c}{{\bf Vertices of Simplex}}\\
		 \hline
		 $A_0$ & $(1, 0, 0, 1)^*$ & $(1, 0, 0, 1)^*$ & $(1, 0, 0, 1)^*$ & $(1, 0, 0, 1)^*$\\
		 $A_1$ & $(1, 0, -1, 0)^*$ & $(1, 0, -1, 0)^*$ & $(1, 0, -1, 0)^*$ & $(1, 0, -1, 0)^*$ \\
		 $A_2$ & $(1,0,\frac{\sqrt{2}}{2},0)$ & $(1,0,\frac{\sqrt{3}}{3},0)$ & $(1,0,\frac{\sqrt{6}}{3},0)$ & $(1,0,\sqrt{\frac{1}{6} \left(\sqrt{5}+3\right)},0)$ \\
		 $A_3$ & $(1,\frac{\sqrt{2}}{4}+\frac{1}{2},\frac{\sqrt{2}}{4}-\frac{1}{2},0)$ & $(1,\frac{1}{4} \left(\sqrt{3}+1\right),\frac{1}{4} \left(\sqrt{3}-1\right),0)$ & $(1,\frac{1}{4} \left(\sqrt{2}+\sqrt{3}\right),\frac{1}{4} \left(\sqrt{6}-1\right),0)$ & $(1,\frac{1}{8} \left(2 \sqrt{3}+\sqrt{5}+1\right),\frac{1}{8} \left(\sqrt{6 \left(\sqrt{5}+3\right)}-2\right),0)$ \\
		 \hline
		 \multicolumn{5}{c}{{\bf The form $\mbox{\boldmath$u$}_i$ of sides opposite $A_i$ }}\\
		\hline
		 $\Bu_0$ & $(0, 0, 0, 1)$ & $(0, 0, 0, 1)$ & $(0, 0, 0, 1)$ & $(0, 0, 0, 1)$ \\
		 $\Bu_1$ & $(1,-1,-\sqrt{2},-\sqrt{2})$ & $(1,-1,-\sqrt{3},-1)$ & $(1,-\frac{1}{\sqrt{2}},-\frac{3}{\sqrt{6}},-1)$ & $(1,\frac{1}{2} \left(1-\sqrt{5}\right),\frac{1}{2} \sqrt{3} \left(1-\sqrt{5}\right),-1)$ \\
		 $\Bu_2$ & $(1, -1, 1, -1)$ & $(1,-\sqrt{3},1,-1)$ & $(1,-\sqrt{3},1,-1)$ & $(1,-\sqrt{3},1,-1)$ \\
		 $\Bu_3$ & $(0, 1, 0, 0)$ & $(0, 1, 0, 0)$ & $(0, 1, 0, 0)$ & $(0, 1, 0, 0)$ \\
		 \hline
		 \multicolumn{5}{c}{{\bf Maximal horoball-type parameter $s_i$ for horoball $\mathcal{B}_i$ at $A_i$ }}\\
		\hline
		 $\max s_0 \implies s_1 $ & $s_0 = 0 \implies s_1=3/5 $ & $s_0 = 0 \implies s_1=3/5$ & $s_0 = 0 \implies s_1=3/5$ & $s_0 = 0 \implies s_1=3/5$ \\
		 $\max s_1 \implies s_0 $ & $s_1 = \frac{1}{41} \left(15-16 \sqrt{2}\right) \implies s_0 = \frac{1}{\sqrt{2}} $ & $s_1=\frac{1}{13} \left(3 - 4 \sqrt{3}\right) \implies s_0 = \frac{1}{13} \left(4 \sqrt{3} + 3 \right) $ & $s_1 = \frac{1}{57} \left(15 - 16 \sqrt{6}\right) \implies s_0 = \sqrt{\frac{2}{3}} $ & $s_1 = \frac{-20 \sqrt{3}-5 \sqrt{5}+12 \sqrt{15}+23}{84 \sqrt{3}+85 \sqrt{5}-44 \sqrt{15}-215} \implies s_0 = \frac{1}{181} \left(32 \sqrt{3}-19 \sqrt{5}+28 \sqrt{15}+30\right)$ \\
		 \hline
		 \multicolumn{5}{c}{ {\bf Volumes of optimal horoball pieces $V_i = vol(\mathcal{B}_i \cap \mathcal{F}_{\Gamma})$}}\\
		\hline
		 $V_{\max \mathcal{B}_0} \implies V_{\mathcal{B}_1}$ & $\frac{1}{16} \left(2 \sqrt{2}+3\right) \implies \frac{1}{16}$ & $ \frac{1}{24} \left(2 \sqrt{3}+3\right) \implies \frac{1}{16 \sqrt{3}} $ & $\frac{1}{8} \sqrt{5 \sqrt{\frac{2}{3}}+\frac{49}{12}} \implies \frac{1}{16 \sqrt{3}} $ & $\frac{1}{96} \left(9 \sqrt{3}+6 \sqrt{5}+\sqrt{15}+6\right) \implies \frac{1}{16 \sqrt{3}}$ \\
		 $V_{\mathcal{B}_0} \impliedby V_{\max \mathcal{B}_1}$ & $\frac{1}{16} \impliedby \frac{1}{16} \left(2 \sqrt{2}+3\right)$ & $\frac{1}{16 \sqrt{3}}  \impliedby \frac{1}{24} \left(2 \sqrt{3}+3\right)$ & $\frac{1}{16 \sqrt{3}} \impliedby \frac{1}{8} \sqrt{5 \sqrt{\frac{2}{3}}+\frac{49}{12}}$ & $\frac{1}{16 \sqrt{3}}  \impliedby \frac{1}{96} \left(9 \sqrt{3}+6 \sqrt{5}+\sqrt{15}+6\right)$\\
		\hline
		 \multicolumn{5}{c}{ {\bf Densities of horoball pieces $\delta_i = vol(\mathcal{B}_i \cap \mathcal{F}_{\Gamma})$}}\\
		\hline
		 $(\delta_{\max s_3},\delta_{s_2})$ & $(0.853553, 0.146447)$ & $(0.881854,0.118146)$ & $(0.908248,0.0917517)$ & $(0.918187,0.0818125)$\\
		 $(\delta_{s_3},\delta_{\max s_2})$ & $(0.146447, 0.853553)$ & $(0.118146,0.881854)$ & $(0.0917517,0.908248)$ & $(0.0818125, 0.918187)$\\
		\hline
		\multicolumn{5}{c}{ {\bf Optimal Horoball Packing Density} }\\
		\hline
		$\delta_{opt}$ & $0.767195$ & $0.838825$ & $0.747914$ & $0.655381$ \\
		\hline
	\end{tabular}%
}
\caption{Data nonarithmetic Coxeter simplex tilings with two ideal vertices in the Cayley-Klein ball model of radius $1$ centered at $(1,0,0,0)$. An asterisk $^*$ indicates an ideal vertex.}
	\label{table:other}
\end{table}

\end{landscape}

\newpage

%============================================================================%
%                                references                                  %
%============================================================================%

\end{document}